\documentclass[12pt]{amsart}

\usepackage{amsmath,amsthm}
\usepackage{amssymb}
\usepackage{hyperref}
\usepackage{enumitem}

\usepackage{graphicx}

\makeatletter
\@namedef{subjclassname@2020}{%
  \textup{2020} Mathematics Subject Classification}
\makeatother

\usepackage[T1]{fontenc}

\usepackage{latexsym,graphicx}
\usepackage{float} 
\usepackage{pgfplots}
\usepackage{hyphenat}

\newtheorem{theorem}{Theorem}
[section]
\newtheorem{corollary}[theorem]{Corollary}
\newtheorem{lemma}[theorem]{Lemma}
\newtheorem{proposition}[theorem]{Proposition}

\newtheorem*{Lunc}{Luntz Theorem}

\newtheoremstyle{myremark}
  {6pt}   
  {6pt}   
  {\normalfont} 
  {}      
  {\bfseries} 
  {.}     
  {.5em}  
  {}      

\theoremstyle{myremark}
\newtheorem{rem}[theorem]{Remark}

\theoremstyle{plain}

\newcommand{\eps}{\varepsilon}

\numberwithin{equation}{section}
\theoremstyle{definition}
\newtheorem{example}{Example}[section]

\pgfplotsset{compat=1.17} 
\pgfplotsset{soldot/.style={only marks,mark=*, line width=0.2pt, mark size=1pt}}
 
\begin{document}


\baselineskip=17pt

\title[Uniqueness sets with angular density, III]{Uniqueness sets with angular density for spaces of entire functions, III:\\ how to minimize the type. }

\author[A. Kononova]{Anna Kononova}
\address{ \parbox{\textwidth}{
School of Mathematical Sciences\\
Holon Institute of Technology\\
Holon 5810201, Israel
}
}
\email{anya.kononova@gmail.com}

\begin{abstract}
    This note is the third part of our work devoted to uniqueness sets for spaces of entire functions. 
    Given a discrete set $\Lambda$ with angular density $\Delta$ with respect to the order $\rho$, satisfying some  regularity condition, we show that there exists a type-minimizing measure $\Delta_0$  with less than $2\rho$ discrete masses. { For the case $\rho=2$, the value of the critical uniqueness type is found in geometric terms.
    }
    
\end{abstract}


\date{\today}
\maketitle

\section{Introduction}

Let $\mathcal E_{\rho,\sigma}$ be the set of entire functions of type $\sigma$ with respect to the order $\rho>0$, 
 that is, 
for each $\eps>0$ there exists $C_\eps$ such that 
$$|f(z)|\le C_\eps e^{(\sigma+\eps)|z|^\rho}, \;\;\; z\in\mathbb C.$$
By  $\mathcal E_{\rho}$ we denote the set of all entire functions of  order $\rho$ and at most normal type: $\mathcal E_{\rho}:=\bigcup_{\sigma\ge 0} \mathcal E_{\rho,\sigma}$.

Let $\Lambda=(\lambda_k)$ be a discrete sequence in $\mathbb C$, counted with
multiplicities and ordered so that
\[
1\le |\lambda_1|\le |\lambda_2|\le\cdots .
\]

\begin{rem}
Zeros on $\Lambda$ are understood with the multiplicities prescribed by the
sequence. Thus, vanishing on $\Lambda$ means vanishing with at least these
multiplicities, while being a zero set means equality of multiplicities.
\end{rem}

We will use the following notation  $$n_\Lambda(R;\alpha, \beta)=\#\{\lambda_k\in\Lambda:|\lambda_k|<R, \;\;\arg \lambda_k\in(\alpha,\beta)\}$$ 
for the number of points of $
\Lambda$ in a given circular sector.

Given a finite Borel measure $\Delta$ on the interval $[0,2\pi]$, we say that $\Lambda$ has an {\sf angular density $\Delta$  with respect to the order $\rho$} if for any $\alpha$ and $\beta$ except for some countable set $X$ there exists a finite limit 
$$\lim_{r\to\infty}\frac{n_\Lambda(r;\alpha, \beta)}{ r^\rho}=\Delta(\alpha,\beta).$$ 

\begin{rem}\label{mod}
Given $p>0$, we regard points of $[0,2\pi p]$ modulo $2\pi p$, that is,
as points of the quotient space $\mathbb R/(2\pi p\mathbb Z)$ under the
canonical projection. In particular, the endpoints $0$ and $2\pi p$ are
identified. We denote by $d_{2\pi p}$ the quotient distance on
$\mathbb R/(2\pi p\mathbb Z)$ induced by the Euclidean metric on $\mathbb R$.
\end{rem}

 We will say that a set $\Lambda$ with  angular density $\Delta$  with respect to the order $\rho$ is {\sf regular}  (and will write $\Lambda\in AD(\Delta,\rho)$)  if either $\rho\notin \mathbb N$, or, if $\rho\in\mathbb N$, then  the following sums are bounded for $0<R<\infty$:
    $$s(R)=\sum_{\substack{\lambda_k\in\Lambda,\\|\lambda_k|< R}}\frac 1{\lambda_k^\rho}.$$
 We refer to this property as the {\sf Lindel\"of condition}.
It follows from the Lindel\"of condition  (see \cite[Chapter II, Sec.3]{Levin}), 
that the measure $\Delta$ has zero $\rho$-th moment:
\begin{equation}\label{Lindelof}
    \int_0^{2\pi} e^{it\rho}d\Delta(t)= 0.
\end{equation}
In this case we will say that $\Delta$ is a {\sf regular measure}.

 Note that the regularity of $\Delta$  \eqref{Lindelof} does not necessarily imply that every $\Lambda$ with angular density $\Delta$ is regular. Nevertheless, the situation sometimes can be improved by adding a certain "rare" set to $\Lambda$ to make it regular, as the following result shows.

 \begin{Lunc}{\rm \cite{Krivosh,Grishin,Lunc}}
    Let $\rho\in\mathbb N$. For any  discrete set $\Gamma$ with regular angular distribution $\Delta_{\Gamma}$ with respect to the order $\rho$, 
    there exists a set $\Gamma^*\supset\Gamma,$ such that 
$\Delta_{\Gamma^*}=\Delta_{\Gamma}$   and 
\begin{equation*}
\lim_{r\to\infty}\left(\sum_{\substack{\lambda_k\in\Gamma^*,\\|\lambda_k|<r}}\frac1{\lambda^\rho_k}\right)=0.
    \end{equation*}
\end{Lunc}
 Throughout the paper, unless explicitly stated otherwise, we will assume that $\Lambda$ is regular.

A   sequence  $\Lambda=(\lambda_k)\in \mathbb C$ is  called a {\sf zero set} of $\mathcal E_{\rho,\sigma}$, if 
there exists $f\in \mathcal E_{\rho,\sigma}$ such that $f^{-1}\{0\}=\Lambda$.

A   sequence  $\Lambda=(\lambda_k)\in \mathbb C$ is  called a {\sf uniqueness set} of $\mathcal E_{\rho,\sigma}$, if the only function from $\mathcal E_{\rho,\sigma}$ that vanishes on $\Lambda$ is the zero function:
$$f\in\mathcal E_{\rho,\sigma}\;\;\;{\it and}\;\;f|_\Lambda \equiv 0\;\;\Rightarrow \;\;\;f\equiv 0.$$ Otherwise,  it is called a     {\sf nonuniqueness set}.

 This note is the third part of our work devoted to zero and uniqueness sets for spaces of entire functions. The first part \cite{PartI} was concerned with some basic facts on the subject. { For the reader’s convenience, the relevant results from Part I will be briefly recalled in Section \ref{Recall}. In particular, it was proven in \cite[Corollary 1.1, Theorem 1.3]{PartI}  that the fact that a set 
  $\Lambda\in AD(\Delta,\rho)$ is a zero set (or a uniqueness set) for $\mathcal E_{\rho,\sigma}$  depends only on
$\Delta$ and $\rho$, rather than on the particular set $\Lambda$ itself. It allows us to introduce the following definitions.

{We define the {\sf critical zero set type for a measure}  $\Delta$:
$$\sigma_Z(\Delta,\rho)=\inf\left\{\sigma: \exists f\in \mathcal E_{\rho,\sigma},\;\;f^{-1}\{0\}\in AD(\Delta,\rho)\right\},$$
and the {\sf critical uniqueness  type 
for a measure $\Delta$}
$$\sigma_U(\Delta,\rho):=\inf\{\sigma: \exists\Lambda\in AD(\Delta,\rho),\;\; \exists f\in \mathcal E_{\rho,\sigma}\setminus \{0\}: \;f|\Lambda=0\}.$$
}
Later we will see that, in each case, the infimum is attained.
\begin{rem}\label{Rem2}
    Note that, due to \cite[Theorem 1.3]{PartI} (see also Section \ref{Recall} below), the definition of $\sigma_U(\Delta, \rho)$  is equivalent to
$$\sigma_U(\Delta,\rho):=\inf\{\sigma: \forall\Lambda\in AD(\Delta,\rho),\;\; \exists f\in \mathcal E_{\rho,\sigma}\setminus \{0\}: \;f|\Lambda=0\}.$$
\end{rem}
}

{For $\rho=1,$ the value $\sigma_Z(\Delta,1)=\sigma_U(\Delta,1)$ has a simple geometric meaning \cite{Azarin-Giner, Khabibullin2} (see also Section \ref{Recall}).
In this case, for regular $\Delta$, there exists the associated planar convex compact set $I_\Delta$, such that $\sigma_U(\Delta,1)$ equals the circumradius $R_\Delta$ of the $I_\Delta$.
}
 In \cite{PartI}, it was demonstrated  how this geometric approach can be extended to determine the value of critical zero type in geometrical terms for arbitrary $\rho\in\mathbb N$.
In addition, it was established in which cases the quantities $\sigma_Z(\Delta,\rho)$ and $\sigma_U(\Delta,\rho)$ coincide.  

Here, our focus shifts to the case where $\sigma_U(\Delta,\rho)<\sigma_Z(\Delta,\rho).$
Recall (see Section \ref{Recall}), that this is possible if and only if $\rho\in (1/2,1)\cup(1,\infty)$. We will fix a value of $\rho$ within this range and omit it from the notation of the critical zero type and critical uniqueness type:
$$\sigma_U(\Delta):=\sigma_U(\Delta,\rho),\;\sigma_Z(\Delta):=\sigma_Z(\Delta,\rho).$$

  We will say that the measure $\Delta^*$ is {\sf type-minimizing} with respect to the measure $\Delta$, if 
\begin{equation}
    \label{type-min}
\sigma_U(\Delta)=\sigma_Z(\Delta+\Delta^*)=\sigma_U(\Delta+\Delta^*).\end{equation}

\begin{theorem}\label {T1}
    Given a regular measure $\Delta$, 
     there exists a type-minimizing measure 
\begin{equation}
    \label{type-min-Delta}
\Delta^*=\sum_{j=1}^{L} A_j\delta_{\alpha_j},
\end{equation}
where $0\le L<2\rho$, $A_j\ge0$ and $\displaystyle\min_{i\ne j}d_{2\pi}(\alpha_i,\alpha_j)>\dfrac\pi\rho.$   
 \end{theorem}

We prove this theorem in {Section \ref{sec1}}. Moreover, as Example \ref{ex1} shows,  for every $\rho\in (1/2, 1)\cup (1,\infty)$  there exists a regular measure $\Delta_\rho$ such that any type-minimizing purely-discrete measure $\Delta_\rho^*$ has at least  $L=\lceil 2\rho\rceil-1$ mass points:
$$\Delta_\rho^*=\sum_{j=1}^{L} A_j\delta_{\alpha_j},\;\;\;A_j>0,$$
where $\lceil x\rceil$ denotes the ceiling function. Hence, the bound $L<2\rho$ cannot, in general,  be improved.

In {Section \ref{sec2}}, we switch  to a geometric approach.
 We introduce the notion of a locally convex curve
associated with a regular measure $\Delta$, and define its maximal local
circumradius $R^*_{\rm loc}$. We prove that this geometric quantity coincides
with the critical uniqueness type:
\[
\sigma_U(\Delta)=R^*_{\rm loc}.
\]
Moreover, the infimum in the definition of the critical type is attained:
\[
\sigma_U(\Delta)
=
\min\left\{\sigma:\ \exists \Lambda\in AD(\Delta,\rho),\
\exists f\in\mathcal E_{\rho,\sigma}\setminus\{0\}
\text{ such that } f|_\Lambda=0\right\}.
\]
Consequently, a regular set $\Lambda\in AD(\Delta,\rho)$ is a uniqueness set
for $\mathcal E_{\rho,\sigma}$ if and only if
\[
\sigma<R^*_{\rm loc}.
\]

In { Section \ref{sec4}}, we consider several particular cases to provide a geometric explanation of how a discrete type-minimizing measure can be constructed in terms of the locally convex curve 
 $K$.

\section{Preliminaries}

This section collects  several definitions and results which will be used throughout the paper.

\subsection{\texorpdfstring{$\rho$}{rho}-trigonometrically convex functions.}\label{rho-prelim}
A $2\pi$-periodic function $h:\mathbb R\to\mathbb R$ is called  {\sf $\rho$-trigonometrically convex} if  for any $t_1<t_2<t_3$ such that $t_3-t_1<\pi/\rho$, the following inequality holds:
\begin{equation}\label{TC}
     h(t_1)\sin \rho(t_2-t_3)+h(t_2)\sin\rho (t_3-t_1)+h(t_3)\sin \rho(t_1-t_2)\le 0.
\end{equation}
We denote the set of {$\rho$-trigonometrically convex functions}  by $TC_\rho.$
{For $\rho\in\mathbb N$ we also define 
$$T_\rho:=\{A\cos\rho t+B\sin\rho t: A,B\in\mathbb R\}\subset TC_\rho.$$
}

We summarize several basic properties of $\rho$-trigonometrically convex functions. Proofs and further details can be found, for example, in \cite{Levin}. 
\subsection{Basic properties of $\rho$-trigonometrically convex functions.}\label{rho-prelim-prop}
\begin{enumerate}[label={\bf\Alph*.}]
\item
For every $\alpha\in\mathbb R$, every function  $h\in TC_\rho$ has finite left and right derivatives $h'_+(\alpha), h'_-(\alpha)$. Moreover, 
$$h'_-(\alpha) \le h'_+(\alpha).$$
\item
If $h_1, h_2\in TC_\rho$, then  $\max\{h_1(t), h_2(t)\}\in TC_\rho$.
\item For every $h\in TC_\rho$ and every $\alpha\in \mathbb R$
$$h(\alpha-\dfrac{\pi}{2\rho})+h(\alpha+\dfrac{\pi}{2\rho})\ge 0.$$
In particular, 
$$\max(|h|)=\max(h).
$$
\item
A $2\pi$-periodic function $h$ is $\rho$-trigonometrically convex if and only if for all $ 0\le\alpha<\beta\le2\pi$ it holds
\begin{equation}\label{rho_tc}h'_+(\beta)-h'_-(\alpha)+\rho^2\int_\alpha^\beta h\ge0.
\end{equation}
For an interval $I\subset \mathbb R$, we will say that $h:I\to\mathbb R$ is {\sf $\rho$-trigonometrically convex on $I$}, and write $h\in TC_\rho(I)$, if  \eqref{rho_tc} holds for every $\alpha<\beta$ in $I$.

\item
Every $\rho$-trigonometrically convex  function $h$ is associated with a uniquely determined measure $\Delta_h$: 
\begin{equation}
    \label{Delta}
2\pi\rho\Delta_h([\alpha,\beta])=h'_+(\beta)-h'_-(\alpha)+\rho^2\int_\alpha^\beta h.\end{equation}

Conversely, given a regular measure $\Delta$, we define $ h_\Delta\in TC_\rho$ as follows:
\begin{equation}
    \label{h}
\hspace{0.5cm}
h_\Delta(\theta)=
\left\{
\begin{array}{lr}\displaystyle
    \frac{\pi}{\sin\pi\rho}\int_{\theta-2\pi}^\theta\cos\rho(\theta-\varphi-\pi){\rm d}\Delta(\varphi), & \;\;\;  \rho {\notin\mathbb N},\hspace{2cm}
\vspace{3pt}
    \\
      \displaystyle -\int_{\theta-2\pi}^\theta{(\varphi-\theta)}\sin\rho(\varphi-\theta){\rm d}\Delta(\varphi),&  \;\;\;\rho {\in\mathbb N.}\hspace{2cm}
    \end{array}
\right.
\end{equation}
Then (see \cite[Chapter II, Theorem 1, Theorem 2]{Levin}), 
$$\Delta=\Delta_{h_\Delta},\;\;\;\rho>0;$$
$$h(t)=h_{\Delta_h}(t),\;\;\;{\rm if} \;\;\rho\notin \mathbb N;$$
$$h_{\Delta_h}-h\in T_\rho,\;\;\;{\rm if} \;\;\rho\in \mathbb N, $$
So, for  non-integer orders $\rho$ there is one-to-one correspondence between $h\in TC_\rho$ and regular measures $\Delta_h,$ while in the case of integer $\rho$,  the reconstrucion of $h$ is possible only up to an additional trigonometric term. 

\item If the set $E_h:=h^{-1}\{(-\infty,0)\}$, where $h\in TC_\rho$, is non-empty, then it can be represented as
        $$ E_h=\bigcup_{j=1}^l(\alpha_j,\beta_j),$$
         where $h(\alpha_j)=h(\beta_j)=0,$
        $\beta_j-\alpha_j\le\pi/\rho$ and $\alpha_{j+1}-\beta_j>\pi/\rho.$

\end{enumerate}

Note that, if $h$ is the indicator of an entire function $f\in \mathcal E_{\rho,\sigma}$ with $\rho\in\mathbb N$, then adding a trigonometric term  to $h$  corresponds to multiplying $f$ by the corresponding exponential factor $e^{Cz^\rho},C\in\mathbb C$, which does not alter  the zero set of a function. 
{For $\rho\in\mathbb N$, we  say that two  $\rho$-trigonometrically convex functions  $h_1, h_2$  are {\sf equivalent}, $h_1\sim h_2$, if $$h_2-h_1\in T_\rho.$$
 Hence, in particular, if two functions $f,g$ in $ \mathcal E_{\rho,\sigma}$ have the same  zero sets $\mathcal Z_f=\mathcal Z_g,$ it follows from the Hadamard factorization theorem that the corresponding indicators are equivalent:
 $$h_{f}\sim h_{g}.$$

}

We will say that a continuous $2\pi$-periodic function $h$ is {\sf  piecewise $\rho$-trigonometric
} if  for some $L\in\mathbb N$, $0< s_1< \dots < s_L\le 2\pi$, $A_k, \theta_k\in\mathbb R$, and for every $k=1,\dots, L$, it holds
 $$h(t)=A_k\cos\rho(t-\theta_k),\;\;t\in [s_{k-1},s_{k}],$$ 
 $$h'_-(s_k)\ne h'_+(s_k).$$
where  $s_0 = s_L - 2\pi$.
{The points of the set $S := \{s_1,\dots, s_L\}$  will be called the {\sf singular points} of $h$.}

A function $h\in TC_\rho$ that is piecewise $\rho$-trigonometric will be referred to as an
{\sf elementary}
 {\sf $\rho$-trigonometrically convex.}

From \eqref{rho_tc} it follows that a piecewise   $\rho$-trigonometric function $h$  is $\rho$-trigonometrically convex if and only if 
{\begin{equation}\label{h+_}
    h'_+(s)> h'_-(s),\;\;\forall s\in S,\end{equation}  
    where $S$ is the set of singular points of $h$.}
In this case, by \eqref{Delta}, the corresponding measure $\Delta_h$ is a discrete measure with  $L$ point masses:
$$2\pi\rho\Delta_h=\sum_{k=1}^L (h'_+(s_k)- h'_-(s_k))\delta_{s_k}.$$

{\subsection{Geometrical interpretation: case $\rho=1$.}\label{rho=1}
 As already noted in the introduction, the case $\rho=1$ is special because of its useful geometric interpretation. 
 {Given  $h\in TC_1$, 
  the corresponding convex compact set $I_h\subset\mathbb C$ is defined by $$I_h:=\bigcap_{t\in[0,2\pi]} \{(x,y):x\cos t+y\sin t\le h(t)\}.$$
  The correspondence between $h\in TC_1$ and  $I_h$  is  one-to-one. {Moreover, if $h_2(t)-h_1(t)\in T_1,$
 then $I_{h_2}$ can be obtained from $I_{h_1}$ by a translation by the vector $(a,b)$.}
} }

{
\subsection{Selected Results from Part I}\label{Recall}

Here we briefly recall several definitions and results from \cite{PartI} that will be used below.

For $h\in TC_\rho$, put $M_h=
\{\theta:\;\;h(\theta)=\displaystyle
\max_{t\in\mathbb R}h(t)\}$. 

{For $\rho>0$ 
we call a function $h\in TC_\rho$  {\sf $\rho$-balanced}, if it satisfies condition $0\in { \rm conv}\{e^{it\rho}: t\in M_h\}$, where by ${\rm conv}X,$ we denote the convex hull of the plane set X. For $\rho\notin \mathbb N$ every function $h\in TC_\rho$ is automatically $\rho$-balanced,} while for an integer $\rho$, any function $h\in TC_\rho$ can be made $\rho$-balanced by adding a $\rho$-trigonometric function: 
there exists $\widehat h\sim h$ such that 
$\widehat h$ is   $\rho$-balanced. Put
\begin{equation}
    \label{lb-m}
\widehat h=\begin{cases}
    h,&\rho \notin\mathbb N,\\
   {\rm the \;\;} \rho{\rm-balanced\;\; modification\;\; of \;}h,&\rho \in\mathbb N.
\end{cases}
\end{equation}

{
We call the set $M$ {\sf  locally $\rho$-balanced} if and only if there exists a triple  of points $\{\alpha,\beta,\gamma\}\subset M$ such that
\begin{equation}\label{loc-bal}
0<\beta-\alpha\le\dfrac\pi\rho;\;\;0\le\gamma-\beta<\dfrac\pi\rho;\;\;\gamma-\alpha\ge\dfrac\pi\rho.
    \end{equation}
In particular, the degenerate 
 case when $\gamma=\beta=\alpha+\dfrac\pi\rho$ is possible.}
We call the function $h\in TC_\rho, \;\;\rho>1/2$, {locally $\rho$-balanced} according to the corresponding property of the set $M_{\widehat h}=\{\theta:\;\widehat h(\theta)=\displaystyle\max_{[0,2\pi]}\widehat h(t)\}$.

Given $h\in TC_\rho$, where $\rho$ 
is an integer, we put
\begin{equation}
h^*(\theta):=\max_{j=0,1,\ldots, \rho-1} h\left(\frac {\theta+2\pi j}{\rho}\right),
\end{equation}
and observe that the function $h^*$ is $1$-trigonometrically convex. By $I_{h^*}$ we denote the planar convex compact set, having supporting function $h^*$. By $R_{h^*}$ we denote the circumradius of $I_{h^*}.$

In the following theorem we collect some results from \cite{PartI}
that we will use.

\begin{theorem} \label{PartI}
{\upshape \cite[Theorem~1.1, Theorem~1.3, Corollary~2.1]{PartI}}
Suppose that $\Delta$ is a $\rho$-regular measure and $h=h_\Delta$ is the associated $\rho$-trigonometrically convex function.
   Then \begin{enumerate}
    \item if $\rho$ is non-integer, then
        $$\sigma_Z(\Delta)=\max_{t\in[0;2\pi]}h(t);$$
        \item  if $\rho$ is integer, then
        $$\displaystyle
            \sigma_Z(\Delta)=\min_{k\in T_\rho}\max_{t\in[0;2\pi]}[h(t)+k(t)]=\max_{t\in[0,2\pi]}\bigl(\widehat h(t)\bigr)=R_{h^*};
        $$
       \item for $\rho>0$ $$\displaystyle\sigma_U(\Delta)=\inf_{k\in TC_\rho}\max_{t\in[0,2\pi]}\bigl(h(t)+k(t)\bigr);$$
            
             \item for $\rho\le 1/2$ and for $\rho=1$,
    $$\sigma_U(\Delta)=\sigma_Z(\Delta);$$
    \item  for $\rho >1/2$
  the equality    $\sigma_U(\Delta)=\sigma_Z(\Delta)$ holds if and only if  the function $\widehat h$ is locally $\rho$-balanced.
     \end{enumerate}

\end{theorem}}

\begin{rem}
Let us  note that the critical uniqueness type $\sigma_u$ can be considered also as a function of the corresponding $\rho$-trigonometrically convex function:
$\widehat\sigma_U(h_\Delta):=\sigma_U(\Delta)$. Then it  is continuous with
respect to the uniform norm. Indeed, let $h_1, h_2\in TC_\rho$, and let $k_1\in TC_\rho$, $k_2\in TC_\rho$ be such that
$$\max_{t\in\mathbb R}|h_j(t)+k_j(t)|=\widehat\sigma_U(h_j),\;\;j=1,2,$$
(note that for $h\in TC_\rho \max_{t\in\mathbb R}h(t)=\max_{t\in\mathbb R}|h(t)|)$,
Then
\begin{align*}
\widehat\sigma_U(h_1)&\le\max_{t\in\mathbb R}|h_1(t)+k_2(t)|
\\&\le  \max_{t\in\mathbb R}(|h_1(t)-h_2(t)|+|h_2(t)+k_2(t)|)\le \|h_1-h_2\|_\infty+\widehat\sigma_U(h_2),
\end{align*}
 analogously
$$\widehat\sigma_U(h_2)\le \|h_1-h_2\|_\infty+\widehat\sigma_U(h_1).$$
Hence, 
$$|\widehat\sigma_U(h_1)-\widehat\sigma_U(h_2))\le  \|h_1-h_2\|_\infty.$$
\end{rem}

\section{Existence of a discrete type-minimizing measure}\label{sec1}

\subsection{Reformulation of Theorem \ref{T1}.}\label{widely spaced}

Let $\Delta^*$ be a type-minimizing measure \eqref{type-min-Delta} for a regular measure $\Delta$. The corresponding $\rho$-trigonometrically convex function $h^*$ (or any of them when $\rho\in\mathbb N$) will also be referred to as type-minimizing.

We will say that a finite set $\mathcal A=\{\alpha_1,\dots,\alpha_L\}\subset\mathbb R/2\pi\mathbb Z$
is widely spaced if
\[
\min_{i\ne j} d_{2\pi}(\alpha_i,\alpha_j)>\frac{\pi}{\rho}.
\] In particular, $L< 2\rho.$

The following proposition is simply a reformulation of Theorem \ref{T1} into the language of $\rho$-trigonometrically convex functions (see also \cite{PartI}). It will be used later in Section \ref{Sect.upper}.
\begin{proposition}\label{Cor1}
    For every $h\in TC_\rho$ there exists $k^*\in TC_\rho$ such that
    \begin{itemize} \item  $k^*$ is elementary $\rho$-trigonometrically convex with widely spaced set of singular points;
        \item $\displaystyle \max_{t\in[0,2\pi]} (h(t)+k^*(t))=\min_{k\in TC_\rho}\max_{t\in[0,2\pi]}\left(h(t)+k(t)\right)=\sigma_U(\Delta_h),$\\
        where $\Delta_h$ is defined by \eqref{Delta}.
    \end{itemize}
\end{proposition}

\subsection{Preparatory lemmas.}

\begin{lemma}\label{lemma_h'} Assume that $h\in TC_\rho,$ and $h(0)=0.$
Then for $|t|<\pi/\rho$ we have
$$\rho\cdot h(t)\ge \max(h'_-(0),h'_+(0))\cdot\sin\rho t.$$
\end{lemma}
\begin{proof}
    Let $t\in(0,\pi/\rho)$. Consider $\delta\in(0,t)$ and substitute $t_1=0, t_2=\delta, t_3=t$ into inequality \eqref{TC}:
    $$h(\delta)\sin\rho t-h(t)\sin\rho\delta\le 0.$$
    Then
    $$h'_+(0)\cdot\sin\rho t -\rho\cdot h(t)=\lim_{\delta\to 0+}\frac1\delta\left(h(\delta)\sin \rho t-h(t)\sin\rho \delta\right)\le 0.$$
    Analogously, taking  $t_1=\delta\in(t-\pi/\rho,0), t_2=0, t_3=t$ we get
    $$-h(\delta)\sin\rho t+h(t)\sin\rho\delta\le 0,$$ and hence
    $$
    h'_-(0)\cdot\sin\rho t -\rho\cdot h(t)=\lim_{\delta\to 0-}\frac1\delta\left(h(\delta)\sin \rho t-h(t)\sin\rho \delta\right)\le 0.$$
    The case 
$t\in(-\pi/\rho,0)$ is treated in the same way. 
\end{proof}

\begin{lemma}\label{lemma2}
    For every $h\in TC_\rho$ there exists an elementary $\rho$-trigonometrically convex function   $\tau(z)$, such that 
    $$\tau(z)\le h(z).$$
\end{lemma}

\begin{proof}
{
If $h(t)\ge0$ for all $t\in\mathbb R$, then we can take $\tau=0$.
Assume now that
\[
E_h:=\{t\in\mathbb R:\ h(t)<0\}\ne\varnothing .
\]
By Property F in Subsection \ref{rho-prelim-prop}, and by the periodicity of
$h$, the connected components of $E_h$ can be enumerated as
\[
E_h=\bigcup_{j\in\mathbb Z}(\alpha_j,\beta_j),
\]
where, for some $l\in\mathbb Z$,
\[
\alpha_{j+l}=\alpha_j+2\pi,\qquad
\beta_{j+l}=\beta_j+2\pi,
\]
and
\[
h(\alpha_j)=h(\beta_j)=0,\qquad
\beta_j-\alpha_j\le \frac{\pi}{\rho},\qquad
\alpha_{j+1}-\beta_j>\frac{\pi}{\rho}.
\]
Since $h<0$ on $(\alpha_j,\beta_j)$, we have
\[
h'_+(\alpha_j)\le0,\qquad h'_-(\beta_j)\ge0.
\]

Given $\alpha,\beta$ such that $0<\beta-\alpha\le \pi/\rho$, and
non-negative parameters $A$ and $B$, put
\[
I_{\alpha,\beta}:=
\left[\alpha-\frac{\pi}{\rho},\beta+\frac{\pi}{\rho}\right].
\]
On this interval we define (see Fig.\ref{Picture1})
\begin{equation}
\label{tau}
\tau_{[\alpha,A,\beta,B]}(t)
=\begin{cases}
\displaystyle A\sin\rho(\alpha-t),
&
t\in\left[\alpha-\frac{\pi}{\rho},\alpha\right],
\\[5pt]
\displaystyle
\max\left\{
A\sin\rho(\alpha-t),
B\sin\rho(t-\beta)
\right\},
&
t\in[\alpha,\beta],
\\[5pt]
\displaystyle B\sin\rho(t-\beta),
&
t\in\left[\beta,\beta+\frac{\pi}{\rho}\right].
\end{cases}
\end{equation}
\begin{figure}[H]
\begin{center}

{
\begin{tikzpicture}[declare function={
    g(\x)=\x<pi/4 ? 0 :
    (\x<0.98279372/2+3*pi/4 ? -0.5*cos(2*(\x) r) : 
          (\x< 3*pi/2 ? 0.75*sin(2*(\x) r):
      (\x< 2*pi ?0: 
      (\x<  9*pi/4 ? 0: 
      0 r)))));}]
    \begin{axis}[
      grid=both, 
      grid style={line width=0.1pt, draw=gray!75},
      axis lines=center,
      axis line style={black},
 unit vector ratio = 1 2,   
     xmin=-0.5, xmax=26*pi/12,
     ymin=-0.5, ymax=1.2,
      xtick=\empty,
      ytick=\empty,
      every axis plot/.append style={line width=1pt, color=black},
          ]
        \addplot[ thick, samples at={0,0.05,...,8}] {g(x)}; 
\fill(pi/4,0) circle (2pt) node[below] {\tiny$\displaystyle \alpha\!-\!\frac\pi{\rho}$} ;
\fill(3*pi/4,0) circle (2pt) node[above] {\tiny$\displaystyle \;\;\alpha$} ;

\fill(pi,0) circle (2pt) node[above] {\tiny$\displaystyle \!\beta$} ; 
\fill(3*pi/2,0) circle (2pt) node[below] {\tiny$\displaystyle \beta\!+\!\frac{\pi}{\rho}$}; 
\fill(0,0.5) circle (1pt) node[left] {\tiny$A$};
\fill(0,0.75) circle (1pt) node[left] {\tiny$B$};

\draw [dashed](0,0.5) - - (pi/2,0.5) ;
\draw [dashed](0,0.75) - - (5*pi/4,0.75) ;
   \end{axis}
        \end{tikzpicture}}

\end{center}

   \caption{\small Graph of  $\tau_{[\alpha, A, \beta, B]}(t)$. 
   }\label{Picture1}
   
\end{figure}
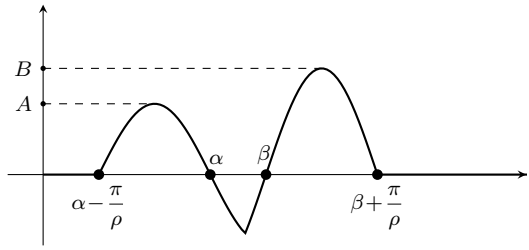

For every $j\in\mathbb Z$, put
\[
A_j:=-\frac{h'_+(\alpha_j)}{\rho}\ge0,
\qquad
B_j:=\frac{h'_-(\beta_j)}{\rho}\ge0,
\]
and
\[
I_j:=I_{\alpha_j,\beta_j}
=
\left[\alpha_j-\frac{\pi}{\rho},
\beta_j+\frac{\pi}{\rho}\right],
\qquad
\tau_j(t):=\tau_{[\alpha_j,A_j,\beta_j,B_j]}(t),
\quad t\in I_j.
\]
Since
\[
\beta_j-\alpha_j\le \frac{\pi}{\rho},
\qquad
\alpha_{j+1}-\beta_j>\frac{\pi}{\rho},
\]
we have
\[
I_j\cap I_m=\varnothing
\qquad\text{whenever } |j-m|\ge2.
\]
Thus every point of $\mathbb R$ belongs to at most two intervals $I_j$.

We define $\tau$ on the whole real line by
\[
\tau(t):=
\begin{cases}
\displaystyle \max\{\tau_j(t): j\in\mathbb Z,\ t\in I_j\},
&
t\in\bigcup_{j\in\mathbb Z} I_j,
\\[5pt]
0,
&
t\notin\bigcup_{j\in\mathbb Z} I_j.
\end{cases}
\]
The maximum is therefore taken over a non-empty finite set of indices.
Moreover, since the family of intervals $I_j$ and the functions $\tau_j$ are
$2\pi$-periodic in $j$, we have
\[
\tau(t+2\pi)=\tau(t).
\]
Thus $\tau$ is a well-defined $2\pi$-periodic function on $\mathbb R$.
(See Fig.\ref{tauP})

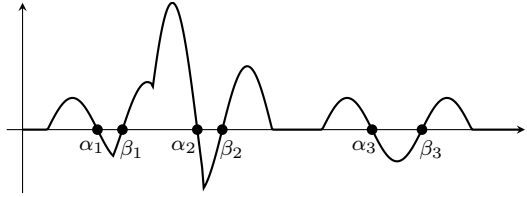
\begin{figure}[H]
\begin{center}
{
\begin{tikzpicture}[declare function={
    g(\x)=\x<pi/4 ? 0 :
    (\x<0.98279372/2+3*pi/4 ? -0.5*cos(2*(\x) r) : 
          (\x< 4.1 ? 0.75*sin(2*(\x) r):
          (\x<5.7 ? -2*cos(2*(\x) r):
      (\x< 5*pi/2? sin(2*(\x) r): 
      (\x<  3*pi ? 0: 
      (\x<  4.5*pi ? 0.5*sin(2*(\x) r):
            0 r)))))));}]
    \begin{axis}[
      grid=both, 
      grid style={line width=0.1pt, draw=gray!75},
      axis lines=center,
      axis line style={black},
 unit vector ratio = 1 2,   
     xmin=-0.5, xmax=5*pi+0.1,
     ymin=-1, ymax=2,
      xtick=\empty,
      ytick=\empty,
      every axis plot/.append style={line width=1pt, color=black},
          ]
        \addplot[ thick, samples at={0,0.05,...,5*pi}] {g(x)}; 

\fill(3*pi/4,0) circle (2pt) node[below] {\tiny$\displaystyle  \!\!\!\alpha_1$} ;
\fill(pi,0) circle (2pt) node[below] {\tiny$\displaystyle \;\;\;\beta_1$} ; 
\fill(7*pi/4,0) circle (2pt) node[below] {\tiny$\displaystyle \!\!\!\!\!\!\!\alpha_2$};
\fill(2*pi,0) circle (2pt) node[below] {\tiny$\displaystyle \;\;\;\beta_2$};
\fill(4*pi,0) circle (2pt) node[below] {\tiny$\displaystyle \;\;\;\beta_3$};
\fill(3.5*pi,0) circle (2pt) node[below] {\tiny$\displaystyle \!\!\!\!\alpha_3$};


   \end{axis}
        \end{tikzpicture}}

\end{center}

   \caption{\small Graph of  $\tau(t)$. 
   }\label{tauP}
   
\end{figure}

Let us show that $\tau\le h$. By the one-sided estimates obtained in the proof
of Lemma \ref{lemma_h'}, applied after shifting the points $\alpha_j$ and
$\beta_j$ to the origin, we have
\[
A_j\sin\rho(\alpha_j-t)\le h(t),
\qquad
t\in\left[\alpha_j-\frac{\pi}{\rho},\beta_j\right],
\]
and
\[
B_j\sin\rho(t-\beta_j)\le h(t),
\qquad
t\in\left[\alpha_j,\beta_j+\frac{\pi}{\rho}\right].
\]
Hence
\[
\tau_j(t)\le h(t),\qquad t\in I_j.
\]
Therefore, if $t\in\bigcup_{j\in\mathbb Z}I_j$, then every function involved
in the maximum defining $\tau(t)$ is bounded above by $h(t)$, and so
\[
\tau(t)\le h(t).
\]
If $t\notin\bigcup_{j\in\mathbb Z}I_j$, then $\tau(t)=0$. In this case
$t\notin E_h$, and hence $h(t)\ge0$. Thus again $\tau(t)\le h(t)$.

It remains to note that $\tau$ is piecewise $\rho$-trigonometric and has only
finitely many singular points modulo $2\pi$. At each singular point, the right derivative is not smaller than the left
derivative. Hence, by the derivative characterization \eqref{rho_tc}, the
function $\tau$ is an elementary $\rho$-trigonometrically convex function.

This proves the lemma.
}

\end{proof}

The next lemma allows us to "erase" extra points of non-smoothness of an elementary $\rho$-trigonometrically convex function. Recall, that such  points are called singular.

\begin{lemma}\label{lemma3}
    Let $\tau$ be an elementary   $\rho$-trigonometrically convex function with a finite set $S$ of singular points. If there are $\alpha,\beta\in S$ such that $0<\beta-\alpha\le\pi/\rho,$ then there exists an elementary   $\rho$-trigonometrically convex function $\kappa$ such that
    \begin{enumerate}
        \item $\kappa(t)=\tau(t)$ for $t\notin [\alpha,\beta]$;
        \item $\kappa(t)\le\tau(t)$ for $t\in [\alpha,\beta]$;
        \item for $\eps>0$ small enough, the function $\kappa$ has at most one singular point on the interval $(\alpha-\eps,\beta+\eps).$ 
    \end{enumerate}
   
\end{lemma}
    \begin{proof}
Suppose that there are no other singular points of $\tau$ between $\alpha$ and $\beta$ (it is sufficient to prove the lemma under this assumption, as the general case clearly follows). Then, for sufficiently small  $\eps\in(0,\frac\pi{2\rho})$, we have
$$\tau(t)=\begin{cases}
    A_1\cos\rho(t-\theta_1),&t\in(\alpha-\eps,\alpha);\\
    A_2\cos\rho(t-\theta_2),&t\in(\alpha,\beta);\\
A_3\cos\rho(t-\theta_3),&t\in(\beta,\beta+\eps),\end{cases}$$
for some real numbers $A_k,\theta_k,\;\;k=1,2,3.$

Put $\tau^*(t):=\tau(t)-A_2\cos\rho(t-
\theta_2).$
Then 
$$\tau^*(t)=\begin{cases}
B_1\sin\rho(\alpha-t),&t\in(\alpha-\eps,\alpha)\\
0,&t\in(\alpha,\beta);\\
B_2\sin\rho(t-\beta),&t\in(\beta,\beta+\eps),
\end{cases}$$
for some real numbers $B_1,B_2.$

The function $\tau^*$ is $\rho$-trigonometrically convex on the interval $(\alpha-\eps,\beta+\eps)$, therefore, we have $B_1>0, B_2>0$.
Hence, the function $B_1\sin\rho(\alpha-t)$ has its minimum at the point $m_1=\alpha+\frac\pi{2\rho},$
and the function $B_2\sin\rho(t-\beta)$ at the point 
$m_2=\beta-\frac\pi{2\rho}.$
 It also follows from the construction that 
 $$m_1-m_2=\alpha-\beta+\frac\pi\rho\in\bigl[0,\frac\pi\rho\bigr).$$

In the case  $m_1=m_2$, that is, when $\beta-\alpha =\dfrac \pi\rho$, we put 
$$\gamma:=\begin{cases}
    \alpha,& {\rm if } \;B_1\ge B_2;
    \\
    \beta,& {\rm if }\; B_1< B_2.
\end{cases}$$
In the case $m_1>m_2$, let $\gamma\in [\alpha,\beta]$ be the abscissa  of the  intersection point of the graphs of the functions $B_1\sin\rho(\alpha-t)$ and $B_2\sin\rho(t-\beta).$

Let us define the function $\kappa^*$ as follows:
$$\kappa^*(t):=\begin{cases}
    \max\{B_1\sin\rho(\alpha-t),B_2\sin\rho(t-\beta)\},&t\in(\alpha,\beta);
    \\
    \tau^*(t),&t\notin(\alpha,\beta).
\end{cases}$$
We have $\kappa^*(t)\le \tau^*(t)$, and the only possible singular point of $\kappa^*$ on the interval $(\alpha-\eps,\beta+\eps)$ is the point $\gamma$. It follows from the construction that $$(\kappa^*)'_+(\gamma)\ge(\kappa^*)'_-(\gamma),$$ therefore, the function $\kappa^*$ is elementary $\rho$-trigonometrically convex. 

Finally, we define the function $\kappa$:
$$\kappa(t):=\kappa^*(t)+A_2\cos\rho(t-\theta_2).$$
The function $\kappa$ is elementary $\rho$-trigonometrically convex, $\kappa(t)\le \tau(t)$ for $t\in[0,2\pi)$, moreover, it has at most one singular point $\gamma$ on the interval $(\alpha-\eps,\beta+\eps).$
 In particular, in the case $\beta=\alpha+\pi/\rho$ and $B_1=B_2$ the function $\kappa$ has no singularity at the point $\gamma$, hence, it has no singularities on the whole interval $(\alpha-\eps,\beta+\eps)$.
 
\end{proof}

Now, from Lemmas \ref{lemma2} and \ref{lemma3} we get
\begin{lemma}\label{lemma4}
       For every $\rho$-trigonometrically convex function $h$ there exists an elementary $\rho$-trigonometrically convex function $\tau$, with a widely spaced set of singular points, such that 
    $$\tau(t)\le h(t).$$
\end{lemma}

It follows from this lemma that the number of singular points of a function $\tau$ us less than $2\rho$.
We should add that this estimate in general cannot be improved, as the following example shows.
\begin{example}  \label{ex1}  
Given $\rho\in(1/2,1)\cup (1,\infty)$ we put $L=\lceil 2\rho\rceil-1<2\rho$, where $\lceil x\rceil$ denotes the ceiling function. That is, $L$ is the largest integer which is less than $2\rho$.
 Let us consider a function $h$ with period $\displaystyle T:=\frac{2\pi}{L}>\frac{\pi}{\rho}$, defined on the interval $[-T/2,T/2]$ by
$$h(t)=\cos\rho t,\;\;\;|t|\le T/2. $$ This function is elementary $\rho$-trigonometrically convex with exactly $L$ singular points $T/2+Tk, \;\;k=0,\dots, L-1$.
Suppose now that there is an elementary $\rho$-trigonometrically convex function $\tau$ such that $\tau(t)\le h(t).$
Then, in particular, for $k\in\mathbb Z,$
$$\tau\left(\pm\frac{\pi}{2\rho}+Tk\right)\le h\left(\pm\frac{\pi}{2\rho}+Tk\right)=0, $$
At the same time, by  Property C in Subsection \ref{rho-prelim-prop}, we have $$\displaystyle\tau\left(\frac{\pi}{2\rho}+Tk\right)+\tau\left(-\frac{\pi}{2\rho}+Tk\right)\ge0.$$ 
Hence,
$$\tau\left(\frac{\pi}{2\rho}+Tk\right)=\tau\left(-\frac{\pi}{2\rho}+Tk\right)=0.$$
On the other hand, for the point  $\displaystyle T/2\in\left[\frac{\pi}{2\rho},T-\frac{\pi}{2\rho}\right]$, we have
$$\tau(T/2)\le h(T/2)<0.$$
Note that the interval $\left[\dfrac{\pi}{2\rho},T-\dfrac{\pi}{2\rho}\right]$ has the length
$T-\dfrac\pi\rho \in\left(0,\dfrac\pi\rho\right).$
Obviously, every elementary $\rho$-trigonometrically convex function $\tau$ such that
$$\tau\left( \frac{\pi}{2\rho}\right)=\tau\left(T-\frac{\pi}{2\rho}\right)=0,$$
while $\tau(T/2)<0$, has a singular point on the interval $\displaystyle\left(\frac{\pi}{2\rho},T-\frac{\pi}{2\rho}\right).$
Analogously, $\tau$ has a singular point on each interval of the form $$\displaystyle\left(\frac{\pi}{2\rho}+Tk,T-\frac{\pi}{2\rho}+Tk\right).$$ Hence, on the whole interval $[0,2\pi)$ the function $\tau$ has at least $L$ singular points.
\end{example}

\subsection{Proof of  Theorem \ref{T1}}

 For a  regular measure $\Delta$  we define the  corresponding $\rho$-trigonometrically convex function $h_\Delta$ by \eqref{h}.
 
By Theorem 
\ref{PartI} (item 3), for every $j\in\mathbb N$ there exists a $\rho$-trigonometrically convex function $k_j$ such that 
 $$\max_{[0,2\pi]}(h_\Delta(t)+k_j(t))<\sigma_U(\Delta)+1/j.$$ 
 By Lemma \ref{lemma4}, for every $j\ge 1$, there exists an elementary $\rho$-trigonometrically convex function $\tau_j$ such that $\tau_j(t)\le k_j(t),$ and the function $\tau_j$ can be chosen to have at most $\lceil 2\rho\rceil-1$ singular points on the interval $[0,2\pi)$. 
Hence, we have a family of elementary $\rho$-trigonometrically convex functions with finite number  of parameters  (not more than $6\rho$).  { Note that, by Property C of $\rho$-trigonometrically convex functions, the amplitudes of all trigonometric "segments" of elementary $\rho$-trigonometrically convex functions in the construction are bounded by $\displaystyle 2\max_{[0,2\pi]}h_\Delta(t)+1$, while their phases and  singular points belong to the interval $[0,2\pi]$.}
Now, from the compactness arguments, it follows that there exists an extremal elementary $\rho$-trigonometrically convex function $k^*$ with at most $\lceil 2\rho\rceil-1$ singular points and such that $$\max_{[0,2\pi]}(h_\Delta(t)+k^*(t))=\sigma_U(\Delta).$$ {
We set $\Delta^*:=\Delta_{k^*}.$
By items (1) and (2) of Theorem \ref{PartI}, applied to the measure
$\Delta+\Delta^*$, we have
$$\sigma_Z(\Delta+\Delta^*)\le \max_{[0,2\pi]}(h_\Delta(t)+k^*(t))=\sigma_U(\Delta).$$
The opposite inequality follows directly from the variational representation.
Indeed, using the variational representation of the critical uniqueness
type, we have
$$
\sigma_U(\Delta)
=
\inf_{k\in TC_\rho}\max_t(h+k)
\le
\inf_{k\in TC_\rho}\max_t(h+k^*+k)
=
\sigma_U(\Delta+\Delta^*).
$$
Since
$
\sigma_U(\Delta+\Delta^*)\le \sigma_Z(\Delta+\Delta^*),
$
we obtain the opposite inequality.
}
\hfill $\Box$

\begin{rem}
     It should be noted that we do not claim that the type-minimizing  measure as in the Theorem \ref{T1} is unique { (see Example \ref{p=3}).}
\end{rem} 

  {
\section{Critical uniqueness type: geometric approach}\label{sec2}
  Recall that in the case  of $f\in\mathcal E_1$,  the indicator  function $h_f$ is the support function of a compact  convex set $I_f\subset \mathbb C$  (see Section \ref{rho=1} ). Such geometric point of view  turns to be useful also in case $\rho\ne 1$ \cite{Bernstein, Levin, Maergoiz}.

  We are starting this section with  a construction of a locally convex curve $K_h$ corresponding to a $\rho$-trigonometrically convex function $h$. In the second subsection we introduce the notion of a {\it nest} — a subarc of  $K_h$ with certain special properties. In our approach, nests play a central role when dealing with the critical uniqueness type and a type-minimizing measure from the geometric point of view.
  The rest of this section is devoted to the geometric description of the value of the critical uniqueness type and to the proof of  Theorem \ref{T2}.

\subsection{Locally  convex curves}

 Given a $\rho$-trigonometrically convex function $h$, let $\widetilde h$ be its   {\sf stretching} defined by $$\widetilde h(t):= h(t/\rho).$$  Since  $ h\in TC_\rho$, its stretching  $\widetilde h$ is a   $2\pi\rho$-periodic function  which is locally  1-trigonometrically convex  (but not necessarily  $2\pi$-periodic), { that is 
\begin{equation}\label{1-TC}
     \widetilde  h(t_1)\sin (t_2-t_3)
     +\widetilde h(t_2)\sin(t_3-t_1)+
     \widetilde h(t_3)\sin (t_1-t_2)\le 0,
\end{equation}
 for any  $t_1<t_2<t_3$ such that  $t_3-t_1<\pi$.
}

Given a  $\rho$-trigonometrically convex function $h$, we will now construct  an object that we will call a {\sf locally convex curve} 
$K_h$.   Such objects have also been used, for example, in \cite{Bernstein, Levin, Maergoiz}. 

For $t\in\mathbb R$ we consider the half-plane
\begin{equation}
    \label{Pi}
\Pi_t:=\{(x,y): x\cos t+y\sin t\le \widetilde h(t)\},\end{equation}
and denote  
its boundary by $$L_t=\partial \Pi_t.$$
We fix an  orientation of the line  $L_t$  by setting $\vec v_t:=(-\sin t, \cos t)$ (so that while traversing   $L_t$ in the direction of $\vec v_t$ we have the interior of $\Pi_t$ to the left).
Given an interval $I\subset \mathbb R$, put
$$\Pi_I:=\bigcap_{t\in I}\Pi_t.
$$
By the definition of $1$-trigonometrically convex functions \eqref{1-TC}, for $\eps~\in~(0,\pi/2)$ we  have
$$\widetilde h(t-\eps)\sin(-\eps)+\widetilde h(t)\sin(2\eps)+\widetilde h(t+\eps)\sin(-\eps)\le 0.$$
Hence,
\begin{equation}
    \label{2hcos}
2\widetilde h(t)\cos \eps \le \widetilde h(t+\eps)+\widetilde h(t-\eps).
\end{equation}
Taking the limit as $\eps\to \pi/2$ and using the continuity of $\widetilde h$, we obtain  Property F (Subsection \ref{rho-prelim-prop}). Therefore, the intersection of all  half-planes corresponding to the points of an interval $[t-\eps,t+\eps], \;\;\eps\in(0,\pi/2]$, is  an unbounded closed convex set:
$$\Pi_t^\eps =\Pi_{[t-\eps, t+\eps]},$$
 and its support function on the interval $[t-\eps, t+\eps]$ is equal to $\widetilde h$.

Let us consider some point $(x^*,y^*)\in L_t.$ Suppose $0<\alpha<\beta\le\pi/2.$ {
Then, it follows from the properties of trigonometrically convex functions that $(x^*,y^*)\in \Pi_{t+\alpha}$ implies $(x^*,y^*)\in \Pi_{t+\beta}.$
Indeed, by \eqref{TC} we have
$$\widetilde h(t)\sin(\alpha-\beta)+\widetilde h(t+\alpha)\sin\beta-\widetilde h(t+\beta)\sin\alpha\le 0.$$
Since $(x^*,y^*)\in L_t\cap\Pi_{t+\alpha},$ it follows that
$$(x^*\cos t+y^*\sin t)\sin(\alpha-\beta)+(x^*\cos(t+\alpha)+y^*\sin(t+\alpha))\sin\beta-\widetilde h(t+\beta)\sin\alpha\le 0,$$
$$\Updownarrow$$
$$\sin\alpha\cdot\left(x^*(\cos t\cos\beta -\sin t\sin\beta) +y^*(\sin t  \cos \beta +\cos t\sin\beta)\right)
-\widetilde h(t+\beta)\sin\alpha\le 0,$$
$$\Updownarrow$$
$$x^*\cos (t+\beta) +y^*\sin(t+\beta)\le\widetilde h(t+\beta).$$
Analogously, $L_t\cap \Pi_{t-\alpha}\subset \Pi_{t-\beta},$ where $0<\alpha<\beta\le\pi/2.$
}
 
Put  $$ K_t:=L_t\cap \Pi_t^{\pi/2}=L_t\cap \Pi_t^\eps, \;\;\forall \eps\in(0,\pi/2].$$
Due to the definition of trigonometrically convex functions, the set $K_t$ is always a connected, nonempty subset of the line $L_t$    (see Fig.\ref{Kt}). From the definition of $K_t$ it follows that ${x,y}\in K_t$ if and only if the following conditions are satisfied:
\begin{eqnarray}
x\cdot \cos t+y\cdot\sin t&=&\widetilde{h}(t)\label {Kt1}
\\
x\cdot \cos (t+\varepsilon)+y\cdot\sin (t+\varepsilon)&\le&\widetilde{h}(t+\varepsilon), \;\;\;0<|\varepsilon|<\dfrac{\pi}{2}.\label{Kt2}
\end{eqnarray}
Subtracting \eqref{Kt1} from \eqref{Kt2} and taking limit while $\varepsilon\to 0$, we get
$$\widetilde h'_-(t)\le -x\sin t+y\cos t\le \widetilde h'_+(t). 
$$
Hence, $K_t$ has the endpoints 
$$A_t=(\widetilde h(t) \cos t-\widetilde h'_-(t) \sin t,\widetilde h'_-(t) \cos t+\widetilde h(t) \sin t),$$ 
and
$$B_t=(\widetilde h(t) \cos t-\widetilde h'_+(t) \sin t,\widetilde h'_+(t) \cos t+\widetilde h(t) \sin t).$$
It follows that the length of the set $K_t$ is equal to
 $$|A_tB_t|=\widetilde h'_+(t)-\widetilde h'_-(t)=2\pi\cdot\Delta(\{t/\rho\}).$$
 
\begin{figure}[h]
    \centering
    \begin{tabular}{cc}
\includegraphics[width=50mm]{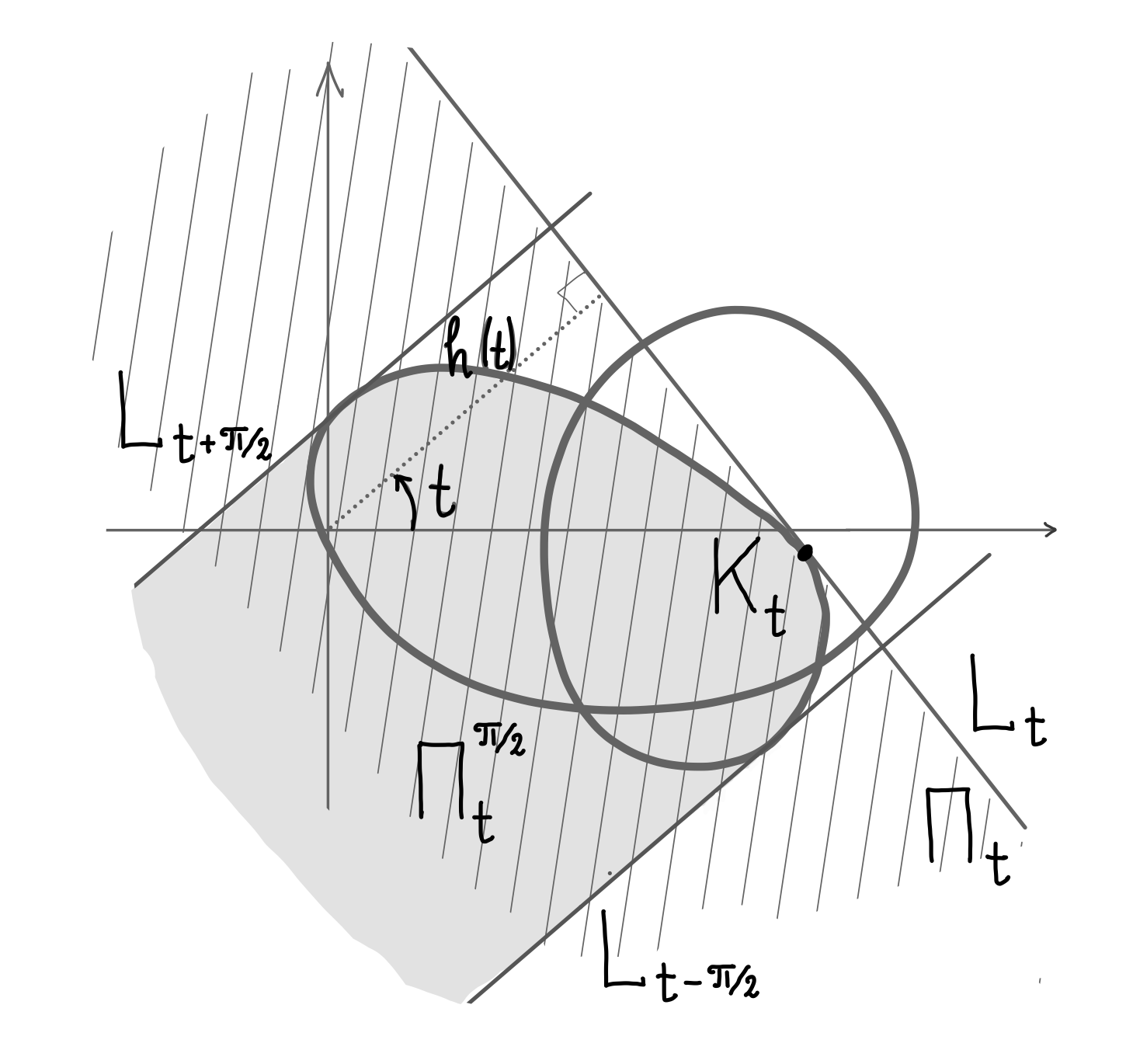}
&
\hspace{50pt}
\includegraphics[width=55mm]{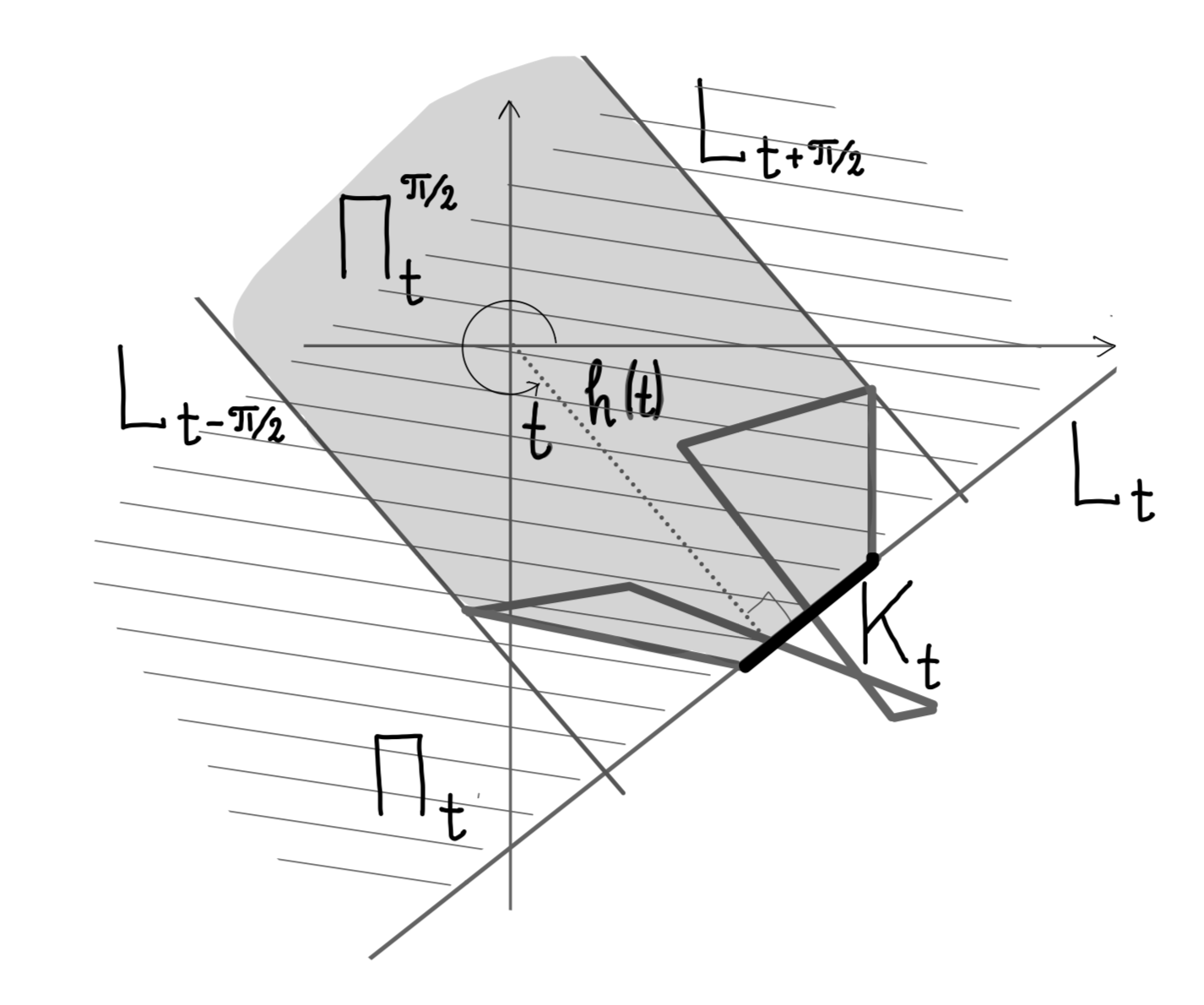}
\end{tabular}
    \caption{\small Construction of the set $K_t$.
}
    \label{Kt}
    
    \end{figure}

Given  $h\in TC_\rho$ we will say that the set 
$$K_h:=\bigcup_{t\in\mathbb R} K_t$$ is the {\sf locally trigonometrically convex curve} associated with $h$. For $t\in\mathbb R$ the set $L_t$ will be called a {\sf support line} of $K_h$.
For any finite interval (closed or open) $I\subset \mathbb R$  we will call the set $$K_I:=\bigcup_{t\in I}K_t$$
a {\sf subarc} of the set $K_h$ (see Fig.\ref{KI1}).
\begin{figure}[h]
    \centering

\includegraphics[width=90mm]{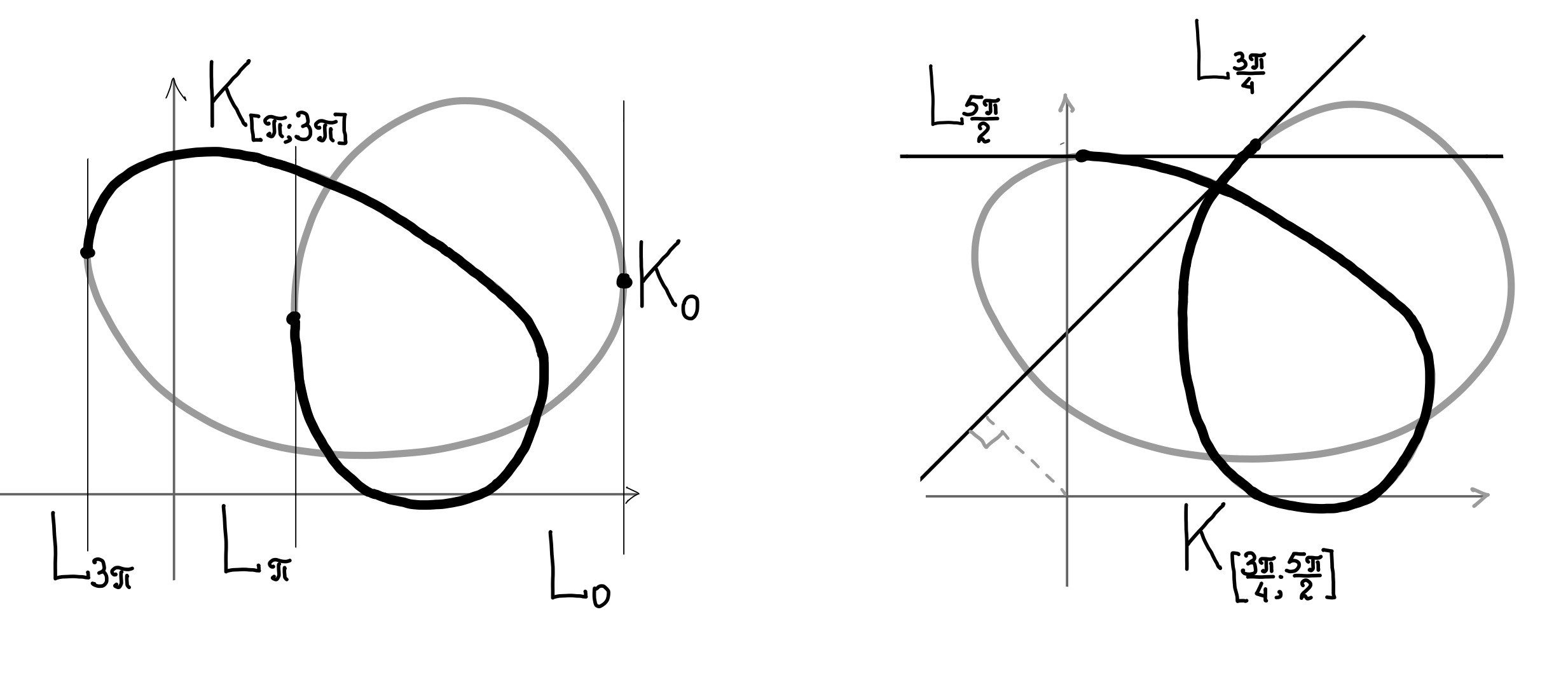}

    \caption{\small Subarcs $K_{[\pi; 3\pi]}$ and $K_{[3\pi/4; 5\pi/2]}$.
}
    \label{KI1}
    
    \end{figure}

Since $\widetilde h$ is $2\pi\rho$-periodic,  $K_h$ is invariant under rotation by the angle $2\pi\rho$.
Note that if $\rho\in\mathbb Q$, then any function $h\in TC_\rho$ has a period that is an integer multiple of $2\pi$,  and hence the corresponding locally convex curve $K_h$ is closed.   On the other hand, in the case $\rho\notin\mathbb Q$ the curve $K_h$ is not closed but remains bounded. 
In the case $\rho\in\mathbb N$, when the curve $K$ is traversed once, the tangent vector $\vec{v}_t$ winds around $\rho$ times, turning through a total angle $2\pi\rho$.



For a bounded subset of the complex plane $E\subset\mathbb C$, let $D(E)$
be the smallest closed disk containing $E$. Its boundary $\partial D(E)$ is
called the {\sf circumcircle} of $E$.
The radius of this circle $R(E)$ is called the {\sf circumradius} of $E$, and its center $O(E)$ is called the {\sf circumcenter} of  the set $E$. {
We will frequently use a simple property of the circumcircle, which follows directly from its definition.  Namely, there are two possibilities:
\\
--- either there are two points $A,B\in E\cap\partial D(E)$ such that $[AB]$
is a diameter of $D(E)$;\\
--- or there are
three points $A,B,C\in E\cap\partial D(E)$
such that the circumcenter $O(E)$ is an inner point of the triangle $ABC$.
}

Consider a $\rho$-trigonometrically convex function $h$ and the corresponding locally
convex curve $K_h$.
 We define {\sf the global circumradius} of  $h$ as the circumradius of the corresponding locally trigonometrically convex curve~$K_h$:
 $$R_h:=R(K_h).$$
 We will also use the 
 following notation: $$D_h=D(K_h), O_h=O(K_h).$$
{

The following  lemma follows immediately from (\ref{Pi}) together with the definition of locally convex curve and the definition of the equivalence of $\rho$-trigonometrically convex functions.

\begin{lemma}\label{shift} Let $\rho\in \mathbb N$.
Equivalent $\rho$-trigonometrically convex functions $h\sim g$ are associated with locally convex curves $K_h$, $K_g$ that can be obtained from each other through a parallel translation. Moreover, if $g(t)=h(t)+a\cos \rho t +b\sin \rho t,$ then $K_{g}$ is a translation of $K_{h}$ by the vector $(a,b).$
\end{lemma}

It follows  that for $\widehat h$, the $\rho$-balanced  modification of $h$ (see   \eqref{lb-m}), the locally convex curve
$K_{\widehat h}$ is a shift of $K_h$. On the other hand,
due to the definition of a $\rho$-balanced function (see Section \ref{Recall}), the global circumradius of $K_{\widehat h}$ equals the maximum of ${\widehat h}$. Hence, 
by items (1) and (2) of Theorem \ref{PartI}, $$\sigma_Z(\Delta)=\max_{t\in[0,2\pi]} \left(\widehat h_\Delta(t)\right)=R_{h_\Delta}.$$
}

}


\subsection{Nests} Now we introduce the central notion of this section.
 We will say that a subarc $N_\alpha:=K_{(\alpha-2\pi,\alpha)},$ $\alpha\in[0,2\pi\rho),$ is a {\sf nest}, if 
  $$
\begin{cases}
       \widetilde h(\alpha)=\widetilde  h(\alpha-2\pi);\\
    \widetilde h'_-(\alpha)\le\widetilde  h'_+(\alpha-2\pi).
\end{cases}$$
In particular, it follows that $K$ has a double-support line $L_\alpha$ such that
$L_\alpha=L_{\alpha-2\pi}$ (see Fig.~
\ref{KI2}).

The following proposition provides equivalent characterizations of a nest.
\begin{proposition}\label{crit_nest}
    Given a regular measure $\Delta$, let $h$ be the corresponding $\rho$-trigonometrically convex function, and let $K$ be the associated locally convex curve. The following assertions are equivalent (see Fig.~
\ref{KI2}):
    \begin{enumerate}
        \item[(i)] $N_\alpha:=K_{(\alpha-2\pi,\alpha)}$ is a nest;
        \vspace{4pt}
        \item[(ii)]  $\displaystyle\int_{\alpha-2\pi}^\alpha e^{i(t-\alpha)} {\rm d}\widetilde\Delta_\alpha(t)\in(\infty,0],$
        where $\widetilde\Delta_\alpha$ is the measure on the interval $(\alpha-2\pi,\alpha)$ defined by $\widetilde\Delta_{\alpha}(E)=\Delta(E/\rho)$ (if $\rho<1$ and $E/\rho$ meets more than one period, $\Delta$ is understood
as its $2\pi$-periodic extension to $\mathbb R$);
          \vspace{4pt}
   \item[(iii)]  the $2\pi$-periodic continuation $k_\alpha$ of the function $\widetilde h|_{[\alpha-2\pi,\alpha)}$ from  the interval $[\alpha-2\pi,\alpha)$ to $\mathbb R$
is the support function of the set $\Pi_{[\alpha-2\pi,\alpha]}={\rm conv} (N_\alpha)$.
   \end{enumerate}
\end{proposition}

\begin{figure}[h]
    \centering
   \includegraphics[width=90mm]{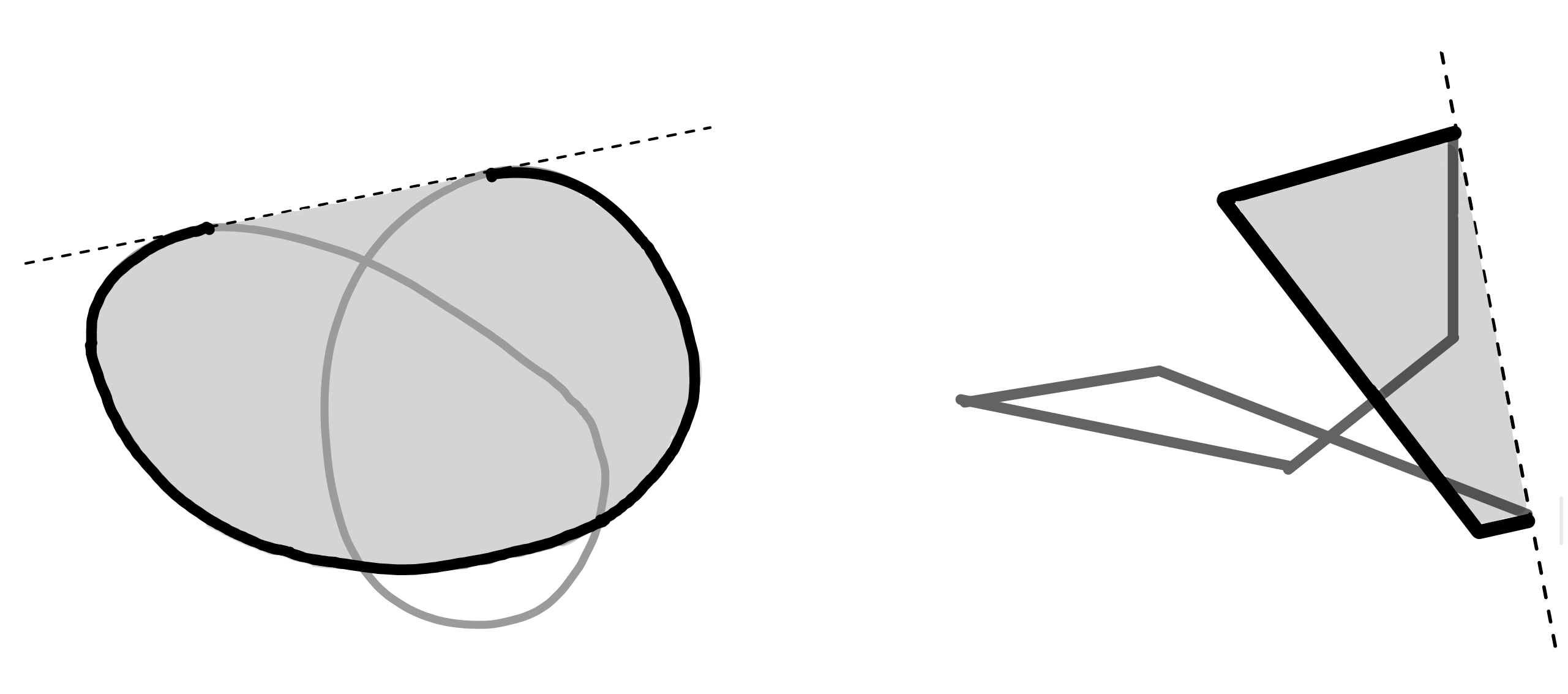}
    \caption{\small Examples of nests, $\rho=2$
}
      \label{KI2}
    
    \end{figure}

\begin{proof}
$(i)\Rightarrow (iii)$. Assume that $N_\alpha$ is a nest and consider the $2\pi$-periodic function defined in statement $(iii)$:
    $$k_\alpha(t+2\pi k)=\widetilde h|_{[\alpha-2\pi,\alpha)}(t).$$
    Since $\widetilde h(\alpha)=\widetilde  h(\alpha-2\pi)$, the function
$k_\alpha$ is continuous on $\mathbb R$.
Next, since $\widetilde h'_-(\alpha)\le\widetilde  h'_+(\alpha-2\pi)$ we have $(k_\alpha)'_-(\alpha)\le(k_\alpha)'_+(\alpha)$, and hence $k_\alpha\in TC_1.$
Therefore, $k_\alpha$ is a support function of some plane convex set. It is clear from the construction that $k_\alpha$ is the support function of $\Pi_{[\alpha-2\pi,\alpha]}$.

Now, let us show that $(iii)\Rightarrow (ii)$.
Note that
$$\partial \Pi_{[\alpha-2\pi,\alpha]}=K_{(\alpha-2\pi,\alpha)}\cup J_\alpha,$$
where $J_\alpha\subset L_\alpha$ is an interval of  length 
$$A_\alpha:=(k_\alpha)'_+(\alpha)-(k_\alpha)'_-(\alpha)=\widetilde h'_+(\alpha-2\pi)-\widetilde  h'_-(\alpha)\ge0.$$
The measure $\Delta_{k_\alpha}$ corresponding to $k_\alpha$  is defined for an arbitrary interval { $[\eta,\zeta)\subset\mathbb R/2\pi\mathbb Z$ by 
$$\Delta_{k_\alpha}[\eta,\zeta)= \widetilde\Delta_\alpha(\eta,\zeta)+\frac1{2\pi}A_\alpha\delta_{\alpha}[\eta,\zeta).$$
}

Since $\Delta_{k_\alpha}$ corresponds to $k_\alpha\in TC_1$, it follows that
$$\displaystyle\int_{\alpha-2\pi}^\alpha e^{it}{\rm d}\Delta_{k_\alpha} (t)=0,$$
which is equivalent to
$$\displaystyle\int_{\alpha-2\pi}^\alpha e^{it} {\rm d}\widetilde\Delta_\alpha(t)+\frac1{2\pi} A_\alpha e^{i\alpha}=0.$$
Hence,
$$\displaystyle\int_{\alpha-2\pi}^\alpha e^{i(t-\alpha)} {\rm d}\widetilde\Delta_\alpha(t)\le0.$$

Finally, if $(ii)$ is satisfied, put
$$A_\alpha:=-2\pi\displaystyle\int_{\alpha-2\pi}^\alpha e^{i(t-\alpha)} {\rm d}\widetilde\Delta_\alpha(t)\ge 0.$$
It means that the measure
$$\widehat\Delta_{\alpha}[\eta,\zeta):=\widetilde \Delta_\alpha[\eta,\zeta)+\frac1{2\pi}A_\alpha\delta_{\alpha}[\eta,\zeta)$$
is a $1$-regular measure. { Using \eqref{Delta},} we observe that the corresponding  $1$-trigonometrically convex function (up to a trigonometric term) is exactly the function 
$k_\alpha$ defined in $(iii)$. Hence,  using the continuity of $k_\alpha$, we have
$$k_\alpha(\alpha)=\widetilde h(\alpha)=\widetilde h(\alpha-2\pi),$$
and
$$A_\alpha:=(k_\alpha)'_+(\alpha)-(k_\alpha)'_-(\alpha)=\widetilde h'_+(\alpha-2\pi)-\widetilde  h'_-(\alpha)\ge0.$$

Hence, $(ii)$ implies $(iii)$ and $(i)$.

\end{proof}

Given a nest $N_\alpha$ we will use the following notation for the parameters of its circumdisk $D_\alpha:=D(N_\alpha)$:
$$R_\alpha:=R(N_\alpha), \;\;O_\alpha:=O(N_\alpha)=r_\alpha e^{i\theta_\alpha}.$$
The value $R_\alpha$ will be referred to as a {\sf local circumradius of $K_h$}.
We have 
\begin{equation} \label{Nest}
    \widetilde h(t)=k_\alpha(t)\le R_\alpha + r_\alpha\cos(t-\theta_\alpha)=:\tau_\alpha(t),\;\;\;t\in[\alpha-2\pi,\alpha],
    \end{equation}
    where the function $\tau_\alpha$ is the support function of $D_\alpha$.

The following lemma will be useful.
\begin{lemma}\label{existence of a nest}
    Let $\beta-\alpha\in \left[0,2\pi/\rho\right),$  
    $h\in TC_\rho\left[\beta-2\pi/\rho,\alpha+2\pi/\rho\right]$, and 
    $$h(t)\le h(\alpha)=h(\beta), \;\;\forall t\in \left[\beta-2\pi/\rho,\alpha+2\pi/\rho\right].$$
Then there exists $t^*\in [\beta, \alpha+2\pi/\rho]$ such that 
$N_{\rho t^*}=K_{(\rho t^*-2\pi, \rho t^*)}$ is a nest.
\end{lemma}

\begin{proof}
Note, that
$$\beta-2\pi/\rho<\alpha\le\beta<\alpha+2\pi/\rho.$$
Let us define $d:[\beta,\alpha+2\pi/\rho]\to \mathbb R$ by
 $$d(t):=h(t-{2\pi}/\rho)-h(t).$$
Then
$$
  d(\alpha+{2\pi}/\rho)=h(\alpha)-h(\alpha+{2\pi}/\rho)\ge0,$$
and 
$$    d(\beta)=h(\beta-2\pi/\rho)-h(\beta)\le 0.
$$
{ From the properties of $\rho$-trigonometrically convex functions it follows that $d$ is continuous and has one-sided derivatives in every point.}
 Therefore, there exists $t^*\in[\beta,\alpha+{2\pi}/\rho]$ such that
 $$d(t^*)=0, d'_-(t^*)\ge 0.$$

Thus, we have 
$$h(t^*)=h(t^*-{2\pi}/{\rho})$$
and 
$$h'_+(t^*-{2\pi}/{\rho})-h'_-(t^*)\ge h'_-(t^*-{2\pi}/{\rho})-h'_-(t^*)
= d'_-(t^*)\ge0.$$
Hence, $N_{\rho t^*}=K_{(\rho t^*-2\pi,\rho t^*)}$ is a nest.
\end{proof}

\subsection{Lower  bound for the critical uniqueness type.}

Put
\begin{equation}
\label{R2}R^*_{\rm loc}:= \sup\{R_\alpha:N_\alpha {\rm \; is \; a \; nest\; for\; }K_h\}.\end{equation}
As we will show in the next section, the supremum in \eqref{R2} is in fact attained, which justifies the term "maximal circumradius" that we have used in the introduction for the value $R^*_{\rm loc}$.

Our next theorem provides a lower bound for the critical uniqueness type. In Theorem \ref{upper} below, we prove that this bound is sharp.

\begin{theorem}
    \label{LN}
 For every   $h\in TC_\rho$ there exists at least one nest for the corresponding locally convex curve $K_h$. Furthermore,
\begin{equation}
\label{R1}
\sigma_U(\Delta_h)\ge R^*_{\rm loc}.
\end{equation}
\end{theorem}

\begin{proof}

 Without loss of generality we suppose that $h(0)=\displaystyle\max_{t\in[0,2\pi\rho]} h(t).$ Then, by Lemma \ref{existence of a nest} (we take $\alpha=\beta=0$), 
    there exists $t^*\in[0,2\pi/\rho]$ such that  $N_{\rho t^*}=K_{(\rho t^*-2\pi,\rho t^*)}$ is a nest.

 Let us now prove inequality \eqref{R1}. Suppose  that for some $\alpha$, such that $N_\alpha$ is a nest, we have $\sigma_U(\Delta_h) < R_\alpha.$ { Then, by definition of the critical uniqueness type, for $\sigma_\alpha: =(\sigma_U(\Delta_h)+R_\alpha)/2$  there exists an entire function } $G\in\mathcal E_{ \rho,\sigma_\alpha}\setminus\{0\}$, and a set $\Lambda \in AD(\Delta_h,\rho)$ such that $G$ vanishes on $\Lambda$  with multiplicities.

 Put
\begin{equation}\label{W}
    W_{\Lambda}(z)=\prod_{\lambda_k\in\Lambda} H\left(\frac{z}{\lambda_k};[\rho]\right), 
    \end{equation}
where 
$$H(w, d) :=(1-w) \exp(w+\frac{w^2}{2}+\ldots+\frac{w^d}{d})$$
is the Hadamard  canonical factor.
The function $W_{\Lambda}(z)$ is an entire function of completely regular growth of order $\rho$  with {\it exact} zero set $\Lambda$. In the case $\rho\notin\mathbb N$, the indicator of $W_\Lambda$ coincides with $h,$ while  for the case $\rho\in\mathbb N$ this indicator $h_{W_\Lambda}$ coincides with $h$ only up to some $\rho$-trigonometric function (see \cite[Chapter\ II, Sect.~1, Theorem 1, Theorem 2]{Levin}):
$$h_{W_\Lambda}(t)=h(t)+a\cos\rho t+b\sin\rho t.$$
The function $G/W_{\Lambda}$ is an entire function of order $\rho$ whose indicator is some $\rho$-trigonometrically convex function $k$.
Applying  Levin's theorem on the indicator of the product of two entire functions \cite[Chapter III, \S 4, Theorem 5]{Levin}, we conclude that 
$$h_G(t):=\limsup_{
r\to\infty}\frac{\log|G(re^{it})|}{r^\rho}=h_{W_{\Lambda}}(t)+k(t).$$
Due to our assumption, we have { $h_G(t)\le\sigma_\alpha<R_\alpha,\;\;t\in [0,2\pi).$} On the other hand,  from  the definition of the nest and of the number $R_\alpha,$ it follows that:

(a) either there are two points $\gamma_1,\gamma_2\in [\alpha-2\pi,\alpha]$  such that $\gamma_2=\gamma_1+\pi$ and
$$\widetilde h(\gamma_j)= R_\alpha +r_\alpha\cos( \gamma_j-\theta_\alpha),$$

(b) or there are
three points $\gamma_j\in[\alpha-2\pi,\alpha],$  with $0\in{\rm conv}\{e^{i\gamma_1},e^{i\gamma_2},e^{i\gamma_3}\},$
such that 
$$\widetilde h(\gamma_j)= R_\alpha +r_\alpha\cos( \gamma_j-\theta_\alpha).$$

Therefore, we have 
$$\widetilde k(\gamma_j)=\widetilde h_G(\gamma_j)-\widetilde h( \gamma_j)-a\cos \gamma_j-b\sin \gamma_j< -r_\alpha\cos( \gamma_j-\theta_\alpha)-a\cos \gamma_j-b\sin \gamma_j,$$
(here, for $\rho\notin \mathbb N$ we put $a=b=0$).

Note that the function 
$\widetilde k^*(t):=\widetilde k(t)+r_\alpha\cos( t-\theta_\alpha)+a\cos t+b\sin t$
is trigonometrically convex on the interval $[\alpha-2\pi,\alpha].$

In case (a), we have $\widetilde k^*(\gamma_1)+\widetilde k^*(\gamma_1+\pi)<0$ which contradicts Property C of trigonometrically convex functions (see Subsection \ref{rho-prelim-prop}).

In case (b), there are three points $\alpha-2\pi\le \gamma_1<\gamma_2< \gamma_3\le\alpha$ such that $0<\gamma_2-\gamma_1<\pi, 0<\gamma_3-\gamma_2<\pi, \gamma_3-\gamma_1>\pi.$ For these points we have
     $\widetilde k^*(\gamma_j)<0.$
     From the fact that   $\widetilde k^*(t)$ is trigonometrically convex, by Property F in Subsection \ref{rho-prelim-prop}, it follows that  if $\widetilde k^*(\gamma_1)<0$ and $\widetilde k^*(\gamma_2)<0$, where $0<\gamma_2-\gamma_1<\pi$, then $$\widetilde k^*(t)<0\;\;\; \forall t\in [\gamma_1,\gamma_2]. $$
     Note that $\gamma_3-\pi\in [\gamma_1,\gamma_2].$ Then we have
     $$\widetilde k^*(\gamma_3-\pi)+\widetilde k^*(\gamma_3)<0,$$
 and again, a contradiction arises. 
The theorem is proved.
    
    \end{proof}

\subsection{Upper bound for the critical uniqueness type.}\label {Sect.upper}

Now, we are going to show  that the lower bound in Theorem \ref{LN} is sharp, and the maximal local circumradius
$R^*_{\rm loc}=\max\{R_\alpha:N_\alpha {\rm \; is \; a \; nest\; for\; }K_h\}$ indeed exists. That will prove Theorem \ref{T2}.

The sharpness of inequality \eqref{R1} follows from the next, slightly more general theorem.

\begin{theorem}
    \label{upper}
    Let $h\in TC_\rho$ be an (upper) indicator of some entire function $F\in \mathcal E_{\rho, \sigma}$:
    $$ h(t)=\limsup_{
r\to\infty}\frac{\log|F(re^{it})|}{r^\rho}.$$
Then there exists a multiplicator $G\in \mathcal E_{\rho}$ such that $F\cdot G\in \mathcal E_{\rho, R_{\rm loc}^*},$ where  $R_{\rm loc}^*$ is the maximal local circumradius  for the corresponding  locally-convex curve $K_h$.
\end{theorem}

\begin{rem}
     Note that  this theorem does not require $F$ to be a function of completely regular growth. We impose no restriction on  the zero distribution of $F$.
\end{rem}

Theorem \ref{upper} follows, in its turn,  from the following proposition.
\begin{proposition}\label{P4}
    Given $h\in TC_\rho$, there exists the maximal local circumradius for the corresponding  locally convex curve $K_h$
     $$R_{\rm loc}^*=\max\{R_\alpha:N_\alpha {\rm \; is \; a \; nest\; for\; }K_h\}.$$  
    Furthermore, there exists a function $\tau\in TC_\rho$, such that 
    $$h(t)+\tau(t)\le R^*_{\rm loc},\;\;\;\forall t\in [0,2\pi].$$    
    \end{proposition}
    
Assuming Proposition \ref{P4}, let us prove Theorem \ref{upper}.
Given $h\in TC_\rho$, the indicator of some $F_h\in \mathcal E_{\rho,\sigma}$, by Proposition \ref{P4} we find $\tau\in TC_\rho$ so that $h(t)+\tau(t)\le R^*_{\rm loc}$.  Then, we  take any function $G$ of completely regular growth with indicator $h_G=\tau$. By the Levin theorem on the indicator of the
product of two entire functions \cite[Chapter III, Sect. 4, Theorem 5]{Levin}, we get the result.

Now we will prove Proposition \ref{P4}. 
\begin{proof}


{
By Proposition \ref{Cor1}, for every $h\in TC_\rho$ there exists an elementary $\rho$-trigonometrically convex function $k^*$ such that 
$$\displaystyle \max_{t\in[0,2\pi]} (h(t)+k^*(t))=\min_{k\in TC_\rho}\max_{t\in[0,2\pi]}\left(h(t)+k(t)\right)=\sigma_U(\Delta),$$ and,
moreover,  the set  of singular points of $k^*$  $\mathcal A=\{\alpha_j\}_{j=1}^L\subset [0,2\pi)$ is widely spaced (see beginning of Subsection \ref{widely spaced}).

\begin{lemma}
    \label{alpha}
  Given $h\in TC_\rho$, let $k^*$ be a corresponding type-minimizing elementary $\rho$-trigonometrically convex function with a set of singular points $\mathcal A$, and let $\Delta$ be the regular measure associated with $h$.  Then, for every $\alpha\in \mathcal A$, we have
    $$h(\alpha)+k^*(\alpha)<\sigma_U(\Delta).$$
     
\end{lemma}

\begin{proof}
    Put 
$$h^*:=h+k^*,$$
\begin{equation}
    \label{m^*}
m^*:=\{t\in [0,2\pi): h^*(t)=\sigma_U(\Delta)\}.\end{equation}

For each point $t\in  m^*$ we have 
$$h^{*'}_-(t)=h^{*'}_+(t)=0,$$
while for each point $\alpha\in \mathcal A$ we have
$$h^{*'}_-(\alpha)=h^{'}_-(\alpha)+k^{*'}_-(\alpha)<h^{'}_+(\alpha)+k^{*'}_+(\alpha)=h^{*'}_+(\alpha).$$
Hence, $\alpha\notin  m^*,\;\forall \alpha\in\mathcal A$.
\end{proof}

Remaining in the context of Lemma \ref{alpha}, let us consider
 a locally convex curve $K^*:=K_{h^*},$ where $h^*=h+k^*$ { as before. For $j=1,\dots, L$ let  us consider a subarc of $K^*$:
$$K^*_j:=K^*_{(\rho\alpha_j,\rho\alpha_{j+1})},$$ where $\mathcal A=\{\alpha_1,\dots,\alpha_L\}$ is the previously defined set  of singular points of $k^*$, and $\alpha_{L+1}=\alpha_1+2\pi$. We will say that a subarc $K^*_j$  is  {\sf anchored}, if there exists a triplet  (where the set $m^*$ is defined by \eqref{m^*}): $$\{\alpha,\beta,\gamma\}\subset (\alpha_j,\alpha_{j+1})\cap m^*=:m^*_j$$ such that}
\begin{equation}\label{anch}
0<\beta-\alpha\le\dfrac\pi\rho;\;\;0\le\gamma-\beta<\dfrac\pi\rho;\;\;\gamma-\alpha\ge\dfrac\pi\rho.
    \end{equation}
In particular, a degenerate 
 case when $\gamma=\beta=\alpha+\dfrac\pi\rho$ is possible.
  
  \begin{rem} We formulate the anchoring condition in the original variable; the corresponding directions on the locally convex curve are $\rho \alpha, \rho\beta,\rho\gamma$.
  \end{rem}


\begin{lemma}
    \label{loc-balK_j}
    Given $h\in TC_\rho$ let $k^*$ be the corresponding type-minimizing elementary $\rho$-trigonometrically convex function, with a widely spaced set of singular points $\mathcal A=\{\alpha_1, \dots, \alpha_L\}\subset [0,2\pi)$.
    Then there exists an anchored subarc of the locally convex curve $K^*:=K_{h+k^*}$:   $$K_j^*=K_{(\rho\alpha_j,\rho\alpha_{j+1})},  $$
     where  $\alpha_{L+1}:=\alpha_1+2\pi$ .\end{lemma}

    \begin{proof}
        
Suppose that for every $j=1,\dots, L$, the subarc $K^*_j$ is not anchored. 
That is, every  set 
$m^*_j$ can be covered by a finite union of {\it closed} disjoint intervals $T^j_{k}:=[\beta^j_{k},\gamma^j_{k}] \subset (\alpha_j,\alpha_{j+1})$: 
 $$m^*_j\subset \bigcup_{k=1}^{N_j}T^j_{k},$$
  such that {(we put $T_k^0=[\beta^0_{k},\gamma^0_{k}]:=T_k^L$, $N_0:=N_L$)} 
  \begin{enumerate}
  \item $\beta^j_{k},\gamma^j_{k}\in m^*_j$, $k=1,\dots, N_j$;
      \item $\gamma^j_{k}-\beta^j_{k}\in[0,\pi/\rho)$,  $k=1,\dots, N_j$;
      \item $\beta^j_{k+1}-\gamma^j_{k}>\pi/\rho$, $k=1,\dots, N_j-1$.
  \end{enumerate}

We will say that  $T_k^j$ is a {\sf well separated  interval} if
\begin{itemize}
    \item[-] either $1<k<N_j$;
    \item[-] or, if $k=1,$ then $\beta^j_1-\gamma^{j-1}_{N_{j-1}}> \pi/\rho$;
    \item[-]  or, if $k=N_j$, then  $\beta_1^{j+1}-\gamma^j_{N_j}> \pi/\rho$.
\end{itemize} 
 If  $\beta^j_1-\gamma^{j-1}_{N_{j-1}}\le \pi/\rho$, we will say that the pair of intervals $T^{j-1}_{N_{j-1}}$ and $T^j_1$ is a {\sf cluster}.

 \begin{rem}
        Note that, despite the similarity between the definitions of a locally $\rho$-balanced set (see Section \ref{Recall}) and the definition of an anchored subarc, Lemma \ref{loc-balK_j} does not follow immediately from  Theorem \ref{PartI}. That is, the fact that the function $h+k^*$ is locally $\rho$-balanced does not imply the fact that the curve $K^*$ has an anchored subarc, since the "balancing set"  may be spread across two subarcs $K^*_j,$ $K^*_{j+1}$.
    \end{rem}

    { Based on  Theorem \ref{PartI} (item (5))  together with Proposition \ref{Cor1} and the definition of the type-minimizing measure \eqref{type-min}, we conclude  that the function $h^*=h+k^*$
is locally $\rho$-balanced.}
Together with the assumption that every subarc $K^*_j$ is not anchored, we conclude that there exists (at least one)  value of $j$, such that the pair of intervals $T^{j-1}_{N_{j-1}}$ and $T^j_1$  is a  cluster.

 Now, let us fix an auxiliary constant $\eps$, which will play a role of "safety cushion" in our construction.
 We choose $\eps$ small enough so that the set $T^\eps$, the $\eps$-neighborhood of the set $\displaystyle T=\bigcup_{j=1}^L\bigcup_{k=1}^{N_j} T_k^j$, consists of $\displaystyle\sum_{j=1}^L N_j$ open intervals $$T^\eps=\bigcup_{j=1}^L\bigcup_{k=1}^{N_j}(\widetilde\beta_k^j;\widetilde\gamma_k^j),$$ such that
 \begin{enumerate}
 \item $(\widetilde\beta_k^j;\widetilde\gamma_k^j)\supset~T_k^j$;
     \item $\widetilde\gamma_k^j-\widetilde\beta_k^j<\pi/\rho,\;\;j=1,\dots, L, k=1.\dots, N_j $;
     \item $\widetilde\beta_{k+1}^j-\widetilde\gamma_k^j>\pi/\rho, \;\;j=1,\dots, N_j-1$;
          \item $\widetilde\beta_1^j-\widetilde\gamma_{N_{j-1}}^{j-1}>\pi/\rho$, if  $T_1^j$ and $T_{N_{j-1}}^{j-1}$ are well separated ;
     \item $\widetilde\beta_1^j-\alpha_j>0$ and $\alpha_j-\widetilde\gamma_{N_{j-1}}^{j-1}>0$.
 \end{enumerate}

Put $\displaystyle H_\eps:=\max_{t\notin  {T^\eps}}h^*(t).$ Clearly, $\sigma_U(\Delta)-H_\eps>0$. Put 
$$d:=\min\left\{\dfrac{\sigma_U(\Delta)-H_\eps}2, \frac{A_1}\rho,\frac{A_2}\rho,\dots, \frac{A_L}\rho\right\}>0,$$
where $A_j:={k_+^{*'}}(\alpha_j)-{k_-^{*'}}(\alpha_j)>0$.

Now we use again the  function $\tau_{[\alpha, A,\beta, B]}$ defined in \eqref{tau}.
For every $j,k$ such that an interval $T_k^j$ is well separated , we put $$\tau_k^j(t):=\tau_{[\widetilde\beta_k^j,d,\widetilde \gamma_k^j,d]}(t).$$ {That is
$$
\tau_k^j(t)\!
\\=\!\begin{cases}
\displaystyle d \sin\rho (\widetilde\beta_k^j-t),&\!t\in{[\widetilde\beta_k^j-\frac\pi{\rho},\widetilde\beta_k^j]};
\\
\max\left\{ \displaystyle d\sin\rho (\widetilde\beta_k^j-t),d\sin\rho (t-\widetilde\gamma_k^j)\right\},& \!t\in[\widetilde\beta_k^j,\widetilde\gamma_k^j];
 \\
 \displaystyle d\sin\rho (t-\widetilde\gamma_k^j),& \!t\in[\widetilde\gamma_k^j,\widetilde\gamma_k^j+\frac{\pi}{\rho}].
 \end{cases}
 $$}
In case when a pair of intervals $T^{j-1}_{N_{j-1}}$ and $T^j_1$  is a cluster, we put
$$b_j:=d\cdot\sin\rho(\alpha_j-\widetilde\gamma_{N_{j-1}}^{j-1})\in(0,d],$$
$$c_j:=d\cdot\sin\rho(\widetilde\beta_1^j-\alpha_j)\in(0, d].
$$

and define
$$\tau_1^j(t):=\left.\tau_{[\widetilde\beta_1^j, b_j,\widetilde\gamma_1^j,d]}\right|_{[\alpha_j; \widetilde\gamma_1^j+\frac{\pi}{\rho}]},$$
that is
$$\tau_1^j(t)\! 
=\!\begin{cases}
\displaystyle b_j \sin\rho (\widetilde\beta_1^j-t),&\!t\in{[\alpha_j,\widetilde\beta_1^j]};
\\
\max\left\{ \displaystyle b_j \sin\rho (\widetilde\beta_1^j-t),d\sin\rho (t-\widetilde\gamma_1^j)\right\},& \!t\in[\widetilde\beta_1^j,\widetilde\gamma_1^j];
 \\
 \displaystyle d\sin\rho (t-\widetilde\gamma_1^j),& \!t\in[\widetilde\gamma_1^j,\widetilde\gamma_1^j+\frac{\pi}{\rho}],
 \end{cases}
 $$
and 
$$\tau_{N_{j-1}}^{j-1}(t):=\left.\tau_{[\widetilde\beta_{N_{j-1}}^{j-1}, d,\widetilde\gamma_{N_{j-1}}^{j-1},c_j]}\right|_{[\widetilde\beta_{N_{j-1}}^{j-1}-\frac{\pi}{\rho}; \alpha_j]}, $$
that is {
$$\tau_{N_{j-1}}^{j-1}(t)\! {
=\!\begin{cases}
\displaystyle d \sin\rho (\widetilde\beta_{N_{j-1}}^{j-1}-t),&\!t\in{[\widetilde\beta_{N_{j-1}}^{j-1}-\frac{\pi}{\rho},\widetilde\beta_{N_{j-1}}^{j-1}]};
\\
\max\left\{ \displaystyle d \sin\rho (\widetilde\beta_{N_{j-1}}^{j-1}-t),c_j\sin\rho (t-\widetilde\gamma_{N_{j-1}}^{j-1})\right\},& \!t\in[\widetilde\beta_{N_{j-1}}^{j-1},\widetilde\gamma_{N_{j-1}}^{j-1}];
 \\
 \displaystyle c_j\sin\rho (t-\widetilde\gamma_{N_{j-1}}^{j-1}),& \!t\in[\widetilde\gamma_{N_{j-1}}^{j-1},\alpha_{j}].
 \end{cases}
 }$$
So, in the cluster case, the auxiliary functions are truncated at their common singular point $\alpha_j$. On the other side they are kept only up to their natural zeroes, namely $\widetilde\gamma_1^j+\frac{\pi}{\rho}$ and $\widetilde\beta_{N_{j-1}}^{j-1}-\frac{\pi}{\rho}$. This prevents these auxiliary functions from entering the regions corresponding to other intervals $T_k^j$, where they could become positive and destroy the strict decrease on $m^*$.

{
Note that for all $1\le j\le L$, $1\le k\le N_j$, we have
\begin{itemize}
\item $\tau_k^j(t)\le 0,
\qquad
\forall t\in T^\eps\cap {\rm Dom}(\tau_k^j),
$
\item $
\tau_k^j(t)<0,
\qquad
t\in T_k^j,
$
\item
$
\tau_k^j(t)\le d,
\qquad
t\in {\rm Dom}(\tau_k^j).
$
\end{itemize}
For every $j=1,\dots,L$, we define the perturbation $\omega_j$ on
$(\alpha_j,\alpha_{j+1})$ as follows. If at a point $t$ there are auxiliary
functions $\tau_k^j$ whose domains contain $t$, we put
$$
\omega_j(t):=
\max\{\tau_k^j(t):1\le k\le N_j,\ t\in {\rm Dom}(\tau_k^j)\}.
$$
If no such function is defined at $t$, we put
$$
\omega_j(t):=0.
$$
Now define
$$
\widehat k(t):=k^*(t)+\omega_j(t),
\qquad
t\in(\alpha_j,\alpha_{j+1}).
$$

By Property B of $\rho$-trigonometrically convex functionc, and by the construction of the functions $\tau_k^j$,
$\widehat k$ is $\rho$-trigonometrically convex on each interval
$(\alpha_j,\alpha_{j+1})$.}

To show that $\widehat k$
is elementary $\rho$-trigonometrically convex on $[0,2\pi)$ we only need to check that $\widehat k'_+(\alpha_j)\ge \widehat k'_-(\alpha_j)$ for $j=1,\dots, L$.

First suppose that the pair of intervals
$T^{j-1}_{N_{j-1}}$ and $T_1^j$ is well separated. Then, by the choice of
$d$, the possible contribution of the auxiliary functions to the derivative jump at
$\alpha_j$ cannot exceed $d\rho$ in the negative direction. Hence
\[
\widehat k'_+(\alpha_j)-\widehat k'_-(\alpha_j)
\ge A_j-d\rho\ge0.
\]

It remains to consider the cluster case. Then, by the construction,
\[
\widehat k'_+(\alpha_j)-\widehat k'_-(\alpha_j)
=
A_j+\left(\tau_1^{j}\right)'_+(\alpha_j)
-\left(\tau_{N_{j-1}}^{j-1}\right)'_-(\alpha_j).
\]
Using the definitions of $b_j$ and $c_j$, we obtain
\begin{align*}
    \widehat k'_+(\alpha_j) - \widehat k'_-(\alpha_j)
    &=
    A_j
    - b_j \rho \cos \rho(\widetilde\beta_1^j - \alpha_j)
    - c_j \rho \cos \rho(\widetilde\gamma_{N_{j-1}}^{j-1} - \alpha_j)
    \\
    &=
    A_j
    - d \rho \sin \rho(\alpha_j - \widetilde\gamma_{N_{j-1}}^{j-1})
      \cos \rho(\widetilde\beta_1^j - \alpha_j)
    \\
    &\quad
    - d \rho \sin \rho(\widetilde\beta_1^j - \alpha_j)
      \cos \rho(\widetilde\gamma_{N_{j-1}}^{j-1} - \alpha_j)
    \\
    &=
    A_j
    - d \rho \sin \rho(\widetilde\beta_1^j
    - \widetilde\gamma_{N_{j-1}}^{j-1})
    \ge 0.
\end{align*}
Thus the derivative jump of $\widehat k$ is non-negative at every point
$\alpha_j$. Hence $\widehat k$ is an elementary
$\rho$-trigonometrically convex function on $[0,2\pi)$.

Since the intervals $T_k^j$ cover $m_j^*$, and since the corresponding
active function $\tau_k^j$ is strictly negative on $T_k^j$, the added
perturbation is strictly negative at every point of $m^*$.
Thus, we have an elementary $\rho$-trigonometrically convex function $\widehat k$ such that
\begin{itemize}
    \item $\widehat k(t)<k^*(t),\;\;\forall t\in m^*$;
      \item $\widehat k(t)\le k^*(t),\;\;\forall t\in T^\eps$;
    \item $\widehat k(t)-k^*(t)\le d\le \dfrac12\left(\sigma_U(\Delta)-\displaystyle\max_{t\notin T^\eps}(h(t)+k^*(t))\right)\;\;\forall t\notin T^\eps$.
    \end{itemize}
  {  It follows that 
    $$\max_{t\in[0,2\pi)}\left(h(t)+\widehat k(t)\right)<\max_{t\in[0,2\pi)}\left(h(t)+ k^*(t)\right),$$
    which contradicts the definition of $k^*$.
This contradiction shows that there exists an anchored subarc $K_j^*.$
 }
}    \end{proof}

Our next aim is to show that each anchored subarc of $K^*$ corresponds to a nest $N_\alpha$ of the locally convex curve $K$  with circumradius $R(N_\alpha)=\sigma_U(\Delta)$.

\begin{lemma}\label{max nest exists}
Given $h\in TC_\rho$, let $k^*$ be a corresponding type-minimizing elementary $\rho$-trigonometrically convex function with widely spaced set of singular points $\mathcal A=\{\alpha_1,\dots,\alpha_L\}$. 
    Suppose that $K^*_j=K^*_{[\rho\alpha_j,\rho\alpha_{j+1}]}$ is an anchored subarc of $K_{h+k^*}$.
    Then there exists $\alpha\in [\rho\alpha_j,\rho\alpha_{j+1}]$ such that  $N_\alpha$ is a nest for $K_h$, { and
$$R(N_\alpha)=\sigma_U(\Delta_h).$$}
\end{lemma}

\begin{proof}

As before, we put $h^*:=h+k^*$.
Assume that $j=1.$ Due to the definition of an elementary $\rho$-trigonometrically convex function, for some $P_j\ge 0, t_j\in[0,2\pi)$, and for all $t\in [\alpha_1-\pi/\rho, \alpha_2+\pi/\rho]$ we have
$$k^*(t)=\begin{cases}
P_0\cos\rho(t-t_0), & t\in[\alpha_1-\pi/\rho,\alpha_1];\\
P_1\cos\rho(t-t_1), & t\in [\alpha_1, \alpha_2];\\
P_2\cos\rho(t-t_2), & t\in [\alpha_2, \alpha_2+\pi/\rho].
\end{cases}
$$
Put 
$$k^*_0(t):=k^*(t)-P_1\cos\rho(t-t_1)=
\begin{cases}
Q_0\sin\rho(\alpha_1-t), & t\in\left[\alpha_1-\dfrac\pi\rho,\alpha_1\right];\\
0, & t\in [\alpha_1, \alpha_2];\\
Q_2\sin\rho(t-\alpha_2), & t\in \left[\alpha_2, \alpha_2+\dfrac\pi\rho\right],
\end{cases}
$$
This function is $\rho$-trigonometrically convex on the interval $\left[\alpha_1-\dfrac\pi\rho, \alpha_2+\dfrac\pi\rho\right]$, hence, $Q_0$ and $Q_2$ are positive, and 
$$
k^*_0(t)\ge 0,\;\; \forall t\in \left[\alpha_1-\frac\pi\rho, \alpha_2+\frac\pi\rho\right].
$$
Therefore, we have
$k^*(t)\ge P_1\cos\rho(t-t_1)\;\;\forall t\in \left[\alpha_1-\dfrac\pi\rho, \alpha_2+\dfrac\pi\rho\right],$
and hence,
\begin{equation}
    \label{k_0>0}h(t)\le \sigma_U(\Delta_h)-P_1\cos\rho(t-t_1),\;\;\forall t\in \left[\alpha_1-\frac\pi\rho, \alpha_2+\frac\pi\rho\right].
    \end{equation}
Since $K^*_1$ is anchored, there exists a triplet  $\{\alpha,\beta,\gamma\}\subset m^*_1$, where $m^*_1= (\alpha_1,\alpha_2)\cap m^*$, satisfying condition \ref{anch}. That is, for any $t\in[0,2\pi)$ we have 
\begin{equation} \label{h^*}
    h^*(t)\le h^*(\alpha)=h^*(\beta)=h^*(\gamma)=\sigma_U(\Delta_h).
\end{equation}
 Note that {$\{\gamma-2\pi/\rho, \alpha+2\pi/\rho\}\subset \left[\alpha_1-\dfrac\pi\rho, \alpha_2+\dfrac\pi\rho\right].$
 }
 
 We proceed in a way analogous to the proof of Lemma \ref{existence of a nest}.
 Put 
 $$d(t):=h\left(t-\frac{2\pi}\rho\right)-h(t).$$
Then, using \eqref{k_0>0} and \eqref{h^*}, we get
\begin{multline*}
    d(\gamma)=h\left(\gamma-\frac{2\pi}\rho\right)-h(\gamma)=h\left(\gamma-\frac{2\pi}\rho\right)-\sigma_U(\Delta_h)+k^*(\gamma)\\
\le
\sigma_U(\Delta_h)-P_1\cos\rho\left(\gamma-\frac{2\pi}\rho-t_1\right)-\sigma_U(\Delta_h)+P_1\cos\rho\left(\gamma-t_1\right)=0.
\end{multline*}
Analogously,
\begin{multline*}
    d\left(\alpha+\frac{2\pi}\rho\right)=h(\alpha)-h\left(\alpha+\frac{2\pi}\rho\right)=\sigma_U(\Delta_h)-k^*(\alpha)-h\left(\alpha+\frac{2\pi}\rho\right)\\
\ge
\sigma_U(\Delta_h)-P_1\cos\rho(\alpha-t_1)-\sigma_U(\Delta_h)+P_1\cos\rho\left(\alpha+\frac{2\pi}\rho-t_1\right)=0.
\end{multline*}

 Therefore, there exists $t^*\in\left[\gamma,\alpha+\dfrac{2\pi}\rho\right]$ such that
 $$d(t^*)=0, d'_+(t^*)\ge 0.$$
Hence, 
$$h(t^*)=h\left(t^*-\dfrac{2\pi}{\rho}\right)$$
and 
$$h'_+\left(t^*-\dfrac{2\pi}{\rho}\right)-h'_-(t^*)\ge h'_+\left(t^*-\dfrac{2\pi}{\rho}\right)-h'_+(t^*)
= d'_+(t^*)\ge0.$$
Hence, $N_{\rho t^*}=K_{(\rho t^*-2\pi,\rho t^*)}$ is a nest. 
{ Note that since $t^*\in\left[\gamma,\alpha+\dfrac{2\pi}\rho\right]$, we have
$$t^*-\dfrac{2\pi}\rho\in\left[\gamma-\dfrac{2\pi}\rho,\alpha\right]\subset \left[\alpha_1-\frac\pi\rho, \alpha_2+\frac\pi\rho\right]. $$ 
Therefore, from \eqref{k_0>0} we see that its circumradius $R(N_{\rho t^*})$ is not greater than $\sigma_U(\Delta_h).$
By \eqref{h^*} and Theorem \ref{LN}, $R(N_{\rho t^*})=\sigma_U(\Delta_h).$

}
\end{proof}

{ Now, given $h\in TC_\rho$, by Lemma \ref{max nest exists} we know that there exists a nest $N_\alpha$ of  $K_h$ such that $R(N_\alpha)=\sigma_U(\Delta_h).$ On the other hand, by Theorem \ref{LN}, this value of the circumradius of a nest is maximal, hence,
  $$R_{\rm loc}^*=\max\{R_\alpha:N_\alpha {\rm \; is \; a \; nest\; for\; }K_h\}=\sigma_U(\Delta_h).$$ 
  Hence, by Proposition \ref{Cor1}, there exists $\tau\in TC_\rho$ such that 
  $$h(t)+\tau(t)\le \sigma_U(\Delta_h)=R_{\rm loc}^*.$$
The proof of Proposition \ref{P4} and Theorem \ref{upper}  is completed.
}}
\end{proof}

Combining Theorem \ref{LN} and Proposition \ref{P4}, we obtain the main
result of this section.

\begin{theorem}\label{T2}
Given a regular measure $\Delta$, we have
\[
\sigma_U(\Delta)=R^*_{\rm loc}
=
\min\left\{\sigma:\ \exists\Lambda\in AD(\Delta,\rho),\
\exists f\in\mathcal E_{\rho,\sigma}\setminus\{0\}: f|_\Lambda=0\right\},
\]
where $R_{\rm loc}^*$ is the maximal local circumradius of the locally convex curve $K$ corresponding to $\Delta$.
\end{theorem}

Consequently, we have the following uniqueness criterion.

\begin{corollary}
A regular set $\Lambda\in AD(\Delta,\rho)$ is a uniqueness set for
$\mathcal E_{\rho,\sigma}$ if and only if
\[
\sigma<R^*_{\rm loc}.
\]
\end{corollary}

By the Luntz theorem, in the case $\rho\in\mathbb N$ the non-uniqueness part
of this criterion remains valid for non-regular sets with regular density.
\begin{corollary}
Let $\rho\in\mathbb N$. If $\Lambda$ is a not necessarily regular set with
regular angular density $\Delta$, then $\Lambda$ fails to be a uniqueness set
for $\mathcal E_{\rho,\sigma}$ whenever
\[
\sigma\ge R^*_{\rm loc}.
\]
\end{corollary}

\section{Type-minimizing measure: geometric construction.}\label{sec4}
Our goal in this section is to show how the locally convex curve $K_h$ can be "fitted" into a circle of radius $R_{\rm loc}^*$. 
In the geometric language Theorem \ref{T1} involves a type of local surgery: we cut the curve in some places, and additional segments are inserted into the cuts in such a way that the local convexity is preserved.
{To illustrate  the counterintuitive idea  that adding segments to a locally convex curve can lead to a reduction of its circumradius, we begin with some examples in the first subsection. In the second subsection  we explain how the type-minimizing measure can be constructed in geometric terms for a particular case $\rho\in(1/2,1)$. The third subsection is devoted to the case $\rho=2$.}

\subsection{Two examples for $\rho\in\mathbb N$.}
\begin{example}
\label{ex}{\rm
Let $\rho =2$, and $2\pi\Delta=\delta_0+\delta_{2\pi/3}+\delta_{-2\pi/3}.$
The corresponding locally convex curve $K_\Delta$ is a triangle  such that its support line  makes a $4\pi/3$-turn (in  the counterclockwise direction) at each vertex (see Fig.~\ref{KI5}). The length of each side of the triangle is $1$, hence the circumradius is $R_\Delta=1/\sqrt 3$.
The three nests in this case correspond to three sides of the triangle, so the maximal local circumradius is $R_{\rm loc}^*=1/2<R_\Delta.$

\begin{figure}[h]
    \centering
   \includegraphics[width=0.5\linewidth]{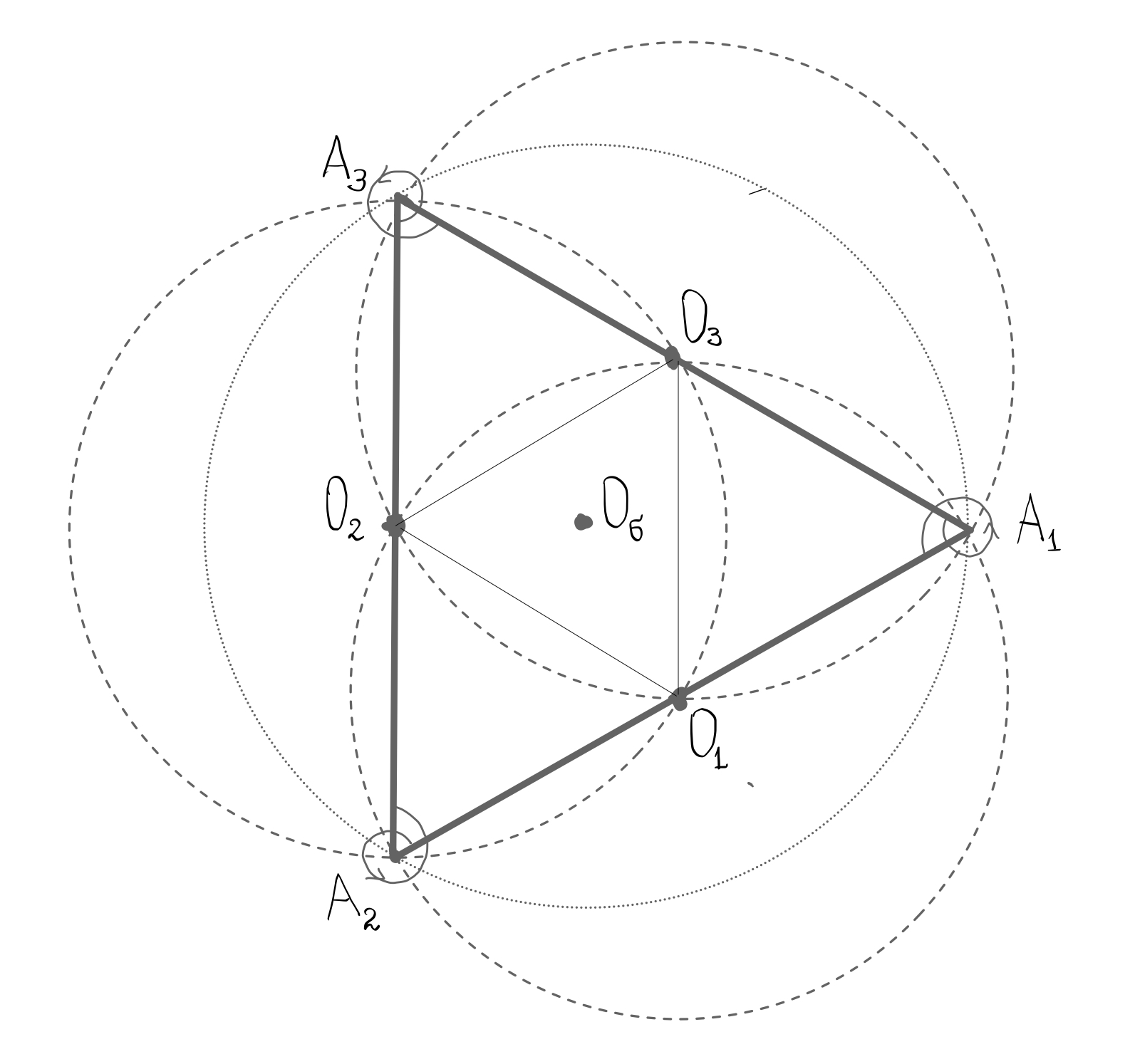}
    \caption{\small $K_\Delta$ for $\rho=2$, $2\pi\Delta=\delta_0+\delta_{2\pi/3}+\delta_{-2\pi/3}.$ The circumradius $R_\Delta = 1/\sqrt 3.$
}
      \label{KI5}
    
    \end{figure}

We will perform a local surgery as follows: we cut $K_\Delta$ at the points $R_\Delta e^{2\pi i},$ $R_\Delta e^{2\pi/3i}$ and $R_\Delta e^{-2\pi/3i}$.  These points correspond to the supporting lines parallel to the lines $(O_1O_2),$ $(O_2O_3),$ $(O_3O_1)$. Then we insert  the segments of the length $|O_1O_2|=|O_2O_3|=|O_3O_1|=1/2$ of the corresponding supporting lines  into the cuts. 

This transformation of the locally convex curve corresponds to the addition of the measure
 $$2\pi\Delta_0:=1/2(\delta_\pi+\delta_{\pi/3}+\delta_{-\pi/3}),$$ so that   $\Delta^*:=\Delta+\Delta_0.$
The corresponding locally convex curve $K_{\Delta^*}$ is  presented in Fig~\ref{KI6}.

\begin{figure}[h]
    \centering
    \includegraphics[width=0.4\linewidth]{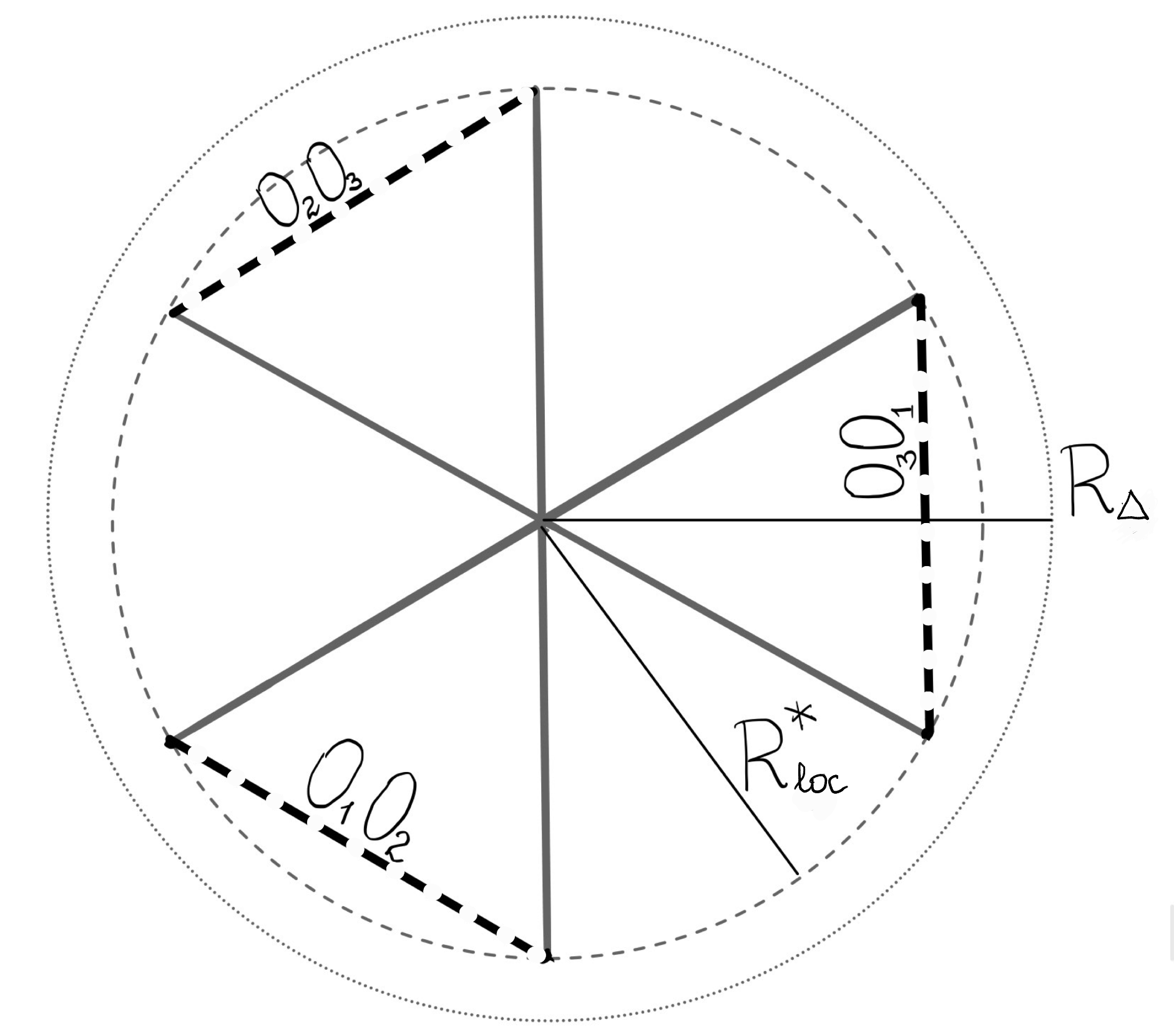}
     \caption{\small After the local surgery, the circumradius is reduced: $R(K_\Delta^*)=R_{\rm loc}^*=1/2<1/\sqrt 3.$ Dashed lines correspond to the masses of the type-minimizing measure.}
    \label{KI6}
\end{figure}

So, we have managed to fit $K_\Delta$ into a smaller-radius circumcircle by adding three supplementary masses. The main point here is to achieve the displacement of the circumcircles of all three nests in such a way they all are covered by the circumdisk of the nest with the  maximal circumradius.
    }
\end{example}

{  In the next example, we show how the idea of displacing the circumcircles of nests can be modified: in some cases, instead of moving all centers to a single point, it suffices to arrange the circumcircles of nests so that they are all contained in a common disk of maximal local circumradius. We construct a family of type-minimizing elementary $3$-trigonometrically convex functions with widely spaced set of singular points.
 \begin{example} \label{p=3}
    Let     $\rho=3$ and $\Delta=\dfrac1{2\pi}\left(\sqrt 3\cdot\delta_0+\delta_{\pi/2}+\sqrt 3\cdot\delta_\pi+\delta_{3\pi/2}\right).$
    Then for $t\in[0,\pi]$ we have
        $$h(t+\pi k)=\cos\left(\left|3t-\dfrac{3\pi}{2}\right|-\dfrac{5\pi}{6}\right),\;\;k\in\mathbb Z.$$
    The corresponding locally convex curve $K_h$ is a rectangle with side  lengths $1$ and $\sqrt 3$, such that its circumradius is $1$ and the support line makes a $3\pi/2$-turn at each vertex (see Fig.~\ref{p=3 1}).

\begin{figure}[h]
    \centering
    \includegraphics[width=0.4\linewidth]{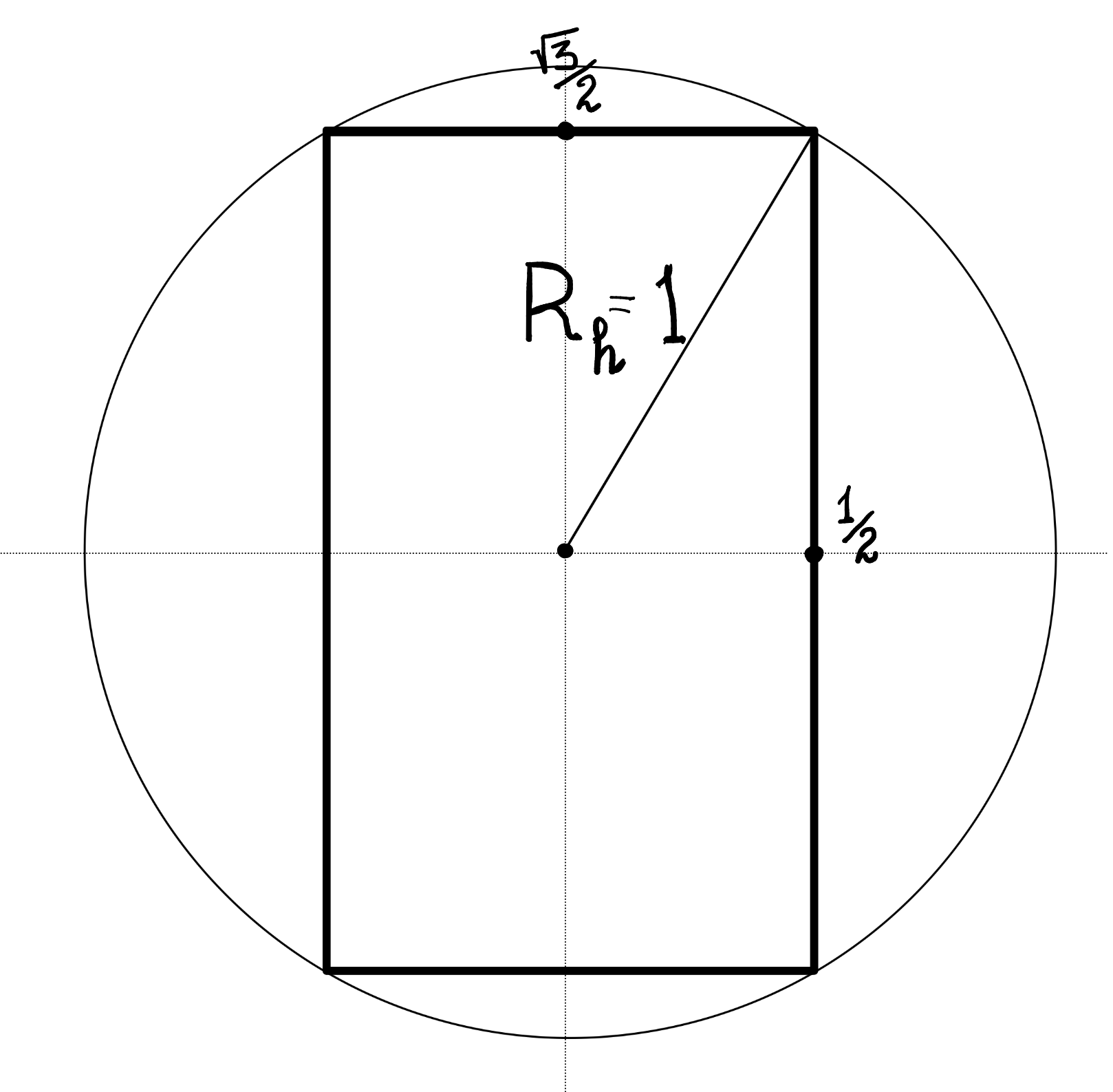}
   \caption{\small $K_h$ corresponding to $\rho=3$,\\ $ \Delta=\dfrac1{2\pi}\left(\sqrt 3\cdot\delta_0+\delta_{\pi/2}+\sqrt 3\cdot\delta_\pi+\delta_{3\pi/2}\right).$}
        \label{p=3 1}
\end{figure}
     We have
    $$\int_{-\pi}^\pi
e^{i(t-\pi)}{\rm d}\Delta|_{(-\pi,\pi)}(t/3)=-\frac{\sqrt 3}{2\pi}\le0,$$
hence, by Proposition \ref{crit_nest} (item ({\it ii})) the subarc $K_{(-\pi,\pi)}$ is a nest. It is easy to see that its circumradius is $\sqrt3/2.$ Hence, $\sigma_U(\Delta)\ge \sqrt 3/2.$

We are going to construct a one-parameter family of type-minimizing measures. Geometrically, the parameter $u$ controls how far the shorter sides of the original rectangle are shifted towards its centre in the modified curve, while the vertical sides are moved to the centre (see Fig.~\ref{p=3 2}).

For every $u\in\left[\dfrac{\sqrt3-\sqrt2}{2},\dfrac{\sqrt3}{2}\right]$ put $\displaystyle\psi=\frac{\arctan( 2u)+\pi/2}3$ and define the measure
$$\Delta_u^*=\frac{\sqrt{u^2+1/4}}{2\pi}\left(\delta_{\psi}+\delta_{\pi-\psi}+\delta_{\pi+\psi}+\delta_{2\pi-\psi}\right),$$
the corresponding $3$-trigonometrically convex function is
$$h_u^*(t+\pi k)=\begin{cases}
    1/2\cos 3t,&t\in[0,\psi];\\
    u\cos3(t-\pi/2),&t\in[\psi,\pi-\psi];
    \\
    -1/2\cos 3t,&t\in[\pi-\psi,\pi].
\end{cases}$$
Then 
$$\max_{t\in[0,2\pi]}(h(t)+h_u^*(t))=\sqrt3/2,$$
where the maximum is attained at the point $t=\pi/6.$ 
Indeed, for $t\in[0,\psi]$, since $\psi\le 5\pi/18<\pi/2$, we have
$$h(t)+h_u^*(t)=\cos(3t-4\pi/6)+1/2\cos 3t=\sqrt3/2\sin 3t\le \sqrt3/2.$$

For $t\in[\psi, \pi/2]$ we have
\begin{gather*}
    h(t)+h_u^*(t)=\cos(3t-4\pi/6)+u\cos 3(t-\pi/2)\\=-1/2\cos 3t+\sqrt3/2\sin3t-u\sin3t\le\sqrt{\dfrac14+\left(\frac{\sqrt3}2-u\right)^2}\le\frac{\sqrt3}2.    
\end{gather*}
By symmetry and periodicity the estimate holds for $t\in[0,2\pi).$
Hence, $\sigma_U(\Delta)=\sqrt3/2.$ (see Fig.~\ref{p=3 2})

        \begin{figure}[h]
            \centering
            \includegraphics[width=0.4\linewidth]{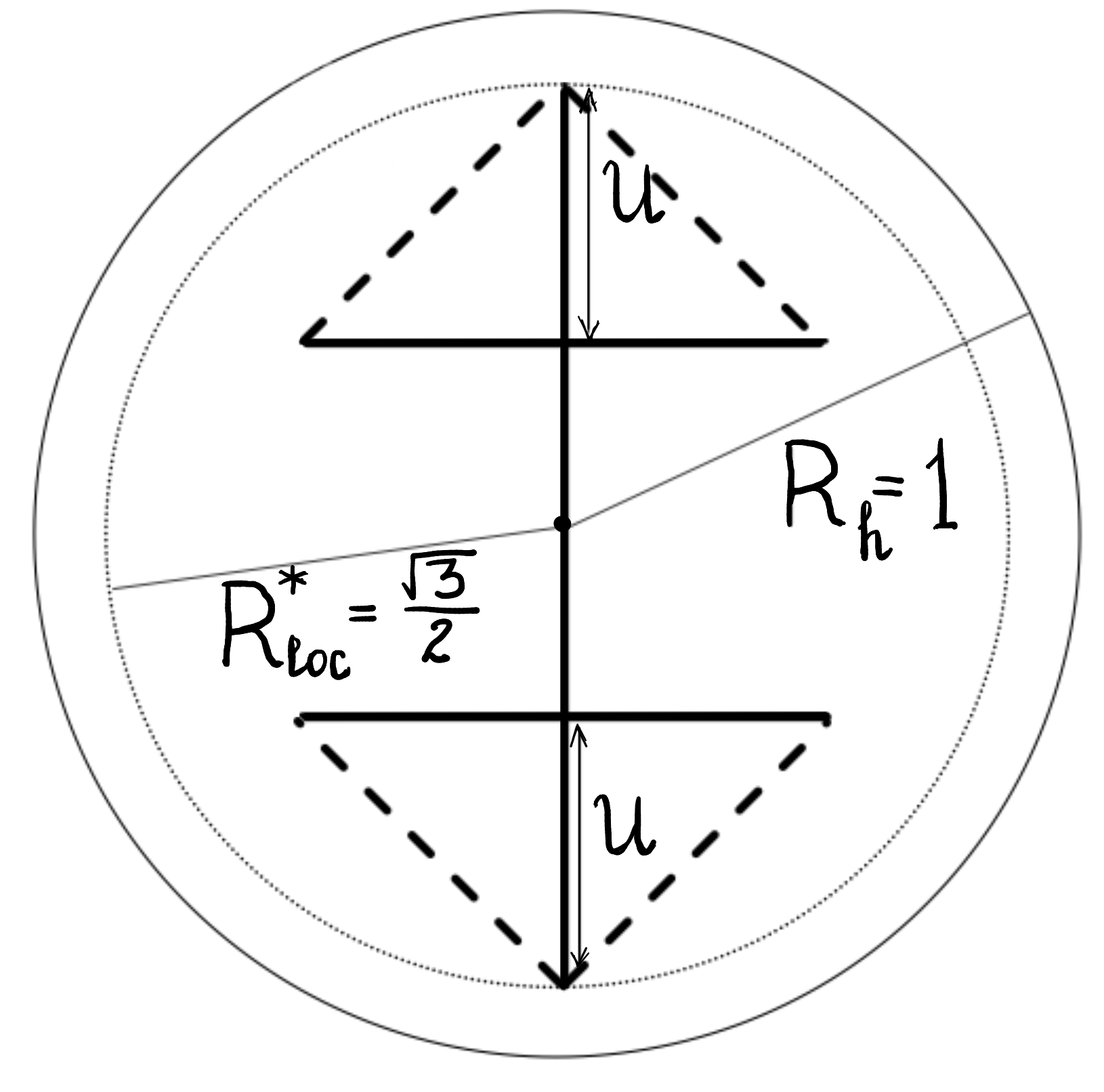}
        \caption{\small Locally trigonometrically convex curve corresponding to $h+h_{1/2}^*.$ (dashed lines correspond to the masses of type-minimizing measure).}
    \label{p=3 2} \end{figure}

Thus, every measure $\Delta_u^*$ in the constructed family  is type-minimizing, and the corresponding function  $h_u^*$
is elementary $3$-trigonometrically convex with a widely spaced set of singular points. We do not claim, however,
that this family exhausts the whole class of type-minimizing measures.

\end{example}
}

\subsection{Construction of a type-minimizing measure: the case $\rho\in(1/2,1)$}
Without loss of generality, we can assume that $h\in TC_\rho$  and $K_h$ is the associated locally-convex curve. Let $h(0)=\displaystyle\max_{t\in\mathbb R} h(t).$  {Note that $\widetilde h(2\pi\rho k)=\widetilde h(0)$ for $k\in\mathbb Z$, since $\widetilde h$ has period $2\pi\rho.$ Put $\gamma_k:=2\pi\rho k.$
Applying Lemma \ref{existence of a nest} to the points $\alpha=-2\pi$ and $\beta=0$, we conclude that there exists $\eta/\rho\in[0,2\pi(1-\rho)/\rho]$ such that
$N_{\eta}=K_{(\eta-2\pi,\eta)}$ is a nest. We put $\eta_k:=\eta+2\pi k\rho,$ { then the sets $N_k^\#:=K_{(\eta_k-2\pi,\eta_k)}$ are also the nests, due to the invariance of $K_h$ under rotation by the angle $2\pi\rho$.}

Suppose  that the function $h$ is not locally $\rho$-balanced, that is, $R_{\rm loc}^*<R_h$.
Since
$\eta-2\pi\le-2\pi\rho< 0\le\eta,$
we have
$$K_{(-2\pi\rho,0)}\subset N_{\eta}=N_0^\#,$$
$$K_{(\gamma_{ k-1}, \gamma_k)}\subset N_{\eta_k}=:N_k^\#, \;\;k\in\mathbb Z.$$ }
Thus, we have a covering of $K_h$ by (a countable number of) nests:
$$K_h\subset\bigcup_{k\in\mathbb Z}  N_k^\#.$$

Next, let $D_0$ be the circumdisk of $N_0^\#$, $O_0=re^{i\theta_0},$  be its circumcenter. Then, due to the invariance of $K_h$ under the rotation, we have 
$$N_k^\#=e^{2\pi\rho k i}N_0^\#,$$
 the circumdisk  $D_k$  of $N_k^\#$ has the center $O_k=re^{i\theta_k},$ where $\theta_k=\theta_{0}+2\pi\rho k$. 

We denote by $R_0\le R_{\rm loc}^* $ the common circumradius of the nests $N_k^\#$ (see Fig.\ref{p1}). 
Now, for every $k\in\mathbb Z$, for $t\in [\gamma_{k-1}, \gamma_k]$ we have
$$\widetilde h(t)\le R_{0}+r\cos(t-\theta_k)=:\tau_{k}(t).$$ 
\begin{figure}[h]
    \centering
\includegraphics[width=0.5\linewidth]{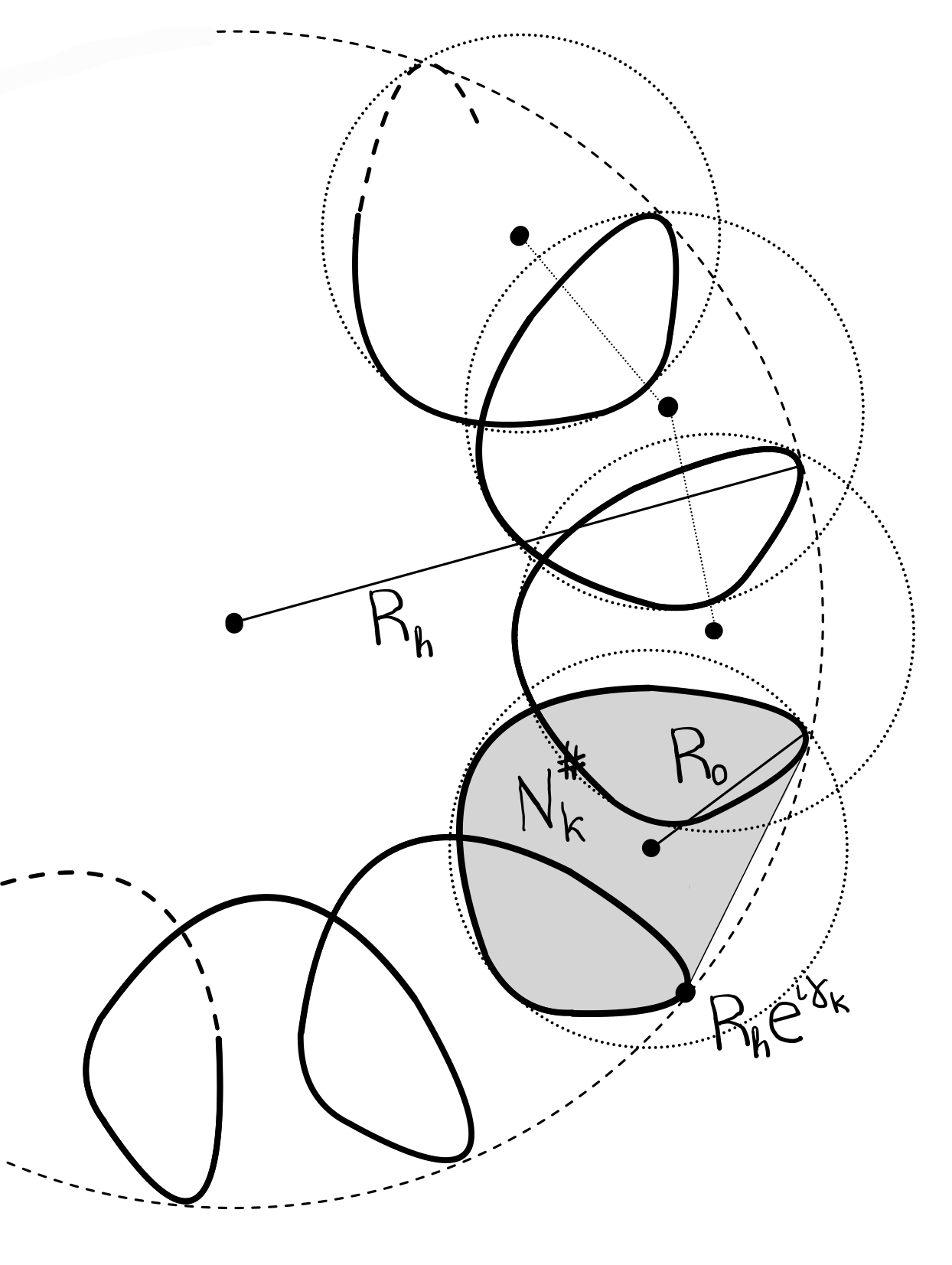}
    \caption{\small $K_h$ corresponding to  $h\in TC_\rho$ for $\rho\in(1/2,1)$.}
    \label{p1}
\end{figure}
Put 
$\zeta_j:=\theta_j+\pi\rho-\pi.$
Then $\tau_{j}(\zeta_j)=\tau_{j+1}(\zeta_j)$. From the geometric point of view, it means that the line $x\cos\zeta_j+y\sin\zeta_j=\tau_{j}(\zeta_j)$ is the tangent line  to $D_{j+1}$   and $ D_j$, simultaneously. Hence, it is parallel to the line $(O_{j} O_{j+1})$. 

Moreover, we have $\tau'_{j}(\zeta_j)>0,\;\;\tau'_{j+1}(\zeta_j)<0,$ and these values do not depend on $j$ due to the rotational invariance of $K_h$.

We put $2\pi\mu:=\tau'_{j}(\zeta_j)-\tau'_{j+1}(\zeta_j)>0.$ Then 
$$2\pi\mu=-r\sin(\zeta_j-\theta_{{j}})+r\sin(\zeta_j-\theta_{{j+1}})=2r\sin\pi\rho=|O_{{j}}O_{j+1}|,$$
hence, $\mu=\displaystyle\frac{|O_{{j}}O_{j+1}|}{2\pi}.$

We put
$$ \widetilde H(t):=
 \tau_j(t), \;\;\; t\in [\zeta_{j-1},\zeta_j].
$$
Then we have $\widetilde h(t)\le \widetilde H(t),\;\;\forall t\in[0,2\pi\rho].$
Define
$$ \widetilde k(t):=
R_0-\widetilde H(t), \;\;\;k(t):=\widetilde k(\rho t),
$$
then $h(t)+k(t)\le R_0$.

Consequently, $\sigma_U(\Delta_h)\le R_0$. Combining this with Theorem
\ref{LN} and the inequality $R_0\le R^*_{\rm loc}$, we obtain
$
R_0=R^*_{\rm loc}.
$

On the other hand, we have
$$\widetilde k(t):=
  - r\cos(t-\theta_j), \;\;\; t\in [\zeta_{j-1},\zeta_j],
$$
and $\widetilde k'_+(\zeta_j)>\widetilde k'_-(\zeta_j)$, hence, $k$ is elementary $\rho$-trigonometrically convex.

To reconstruct the  measure corresponding to $k$, we need only the information about the function on some interval with length $2\pi$.  Let us take the interval $[\theta_0/\rho-\pi,\theta_0/\rho+\pi].$ On this interval there is only one singular point $\zeta_0/\rho$ of the function $k.$

\begin{figure}[h]
    \centering
    \includegraphics[width=0.5\linewidth]{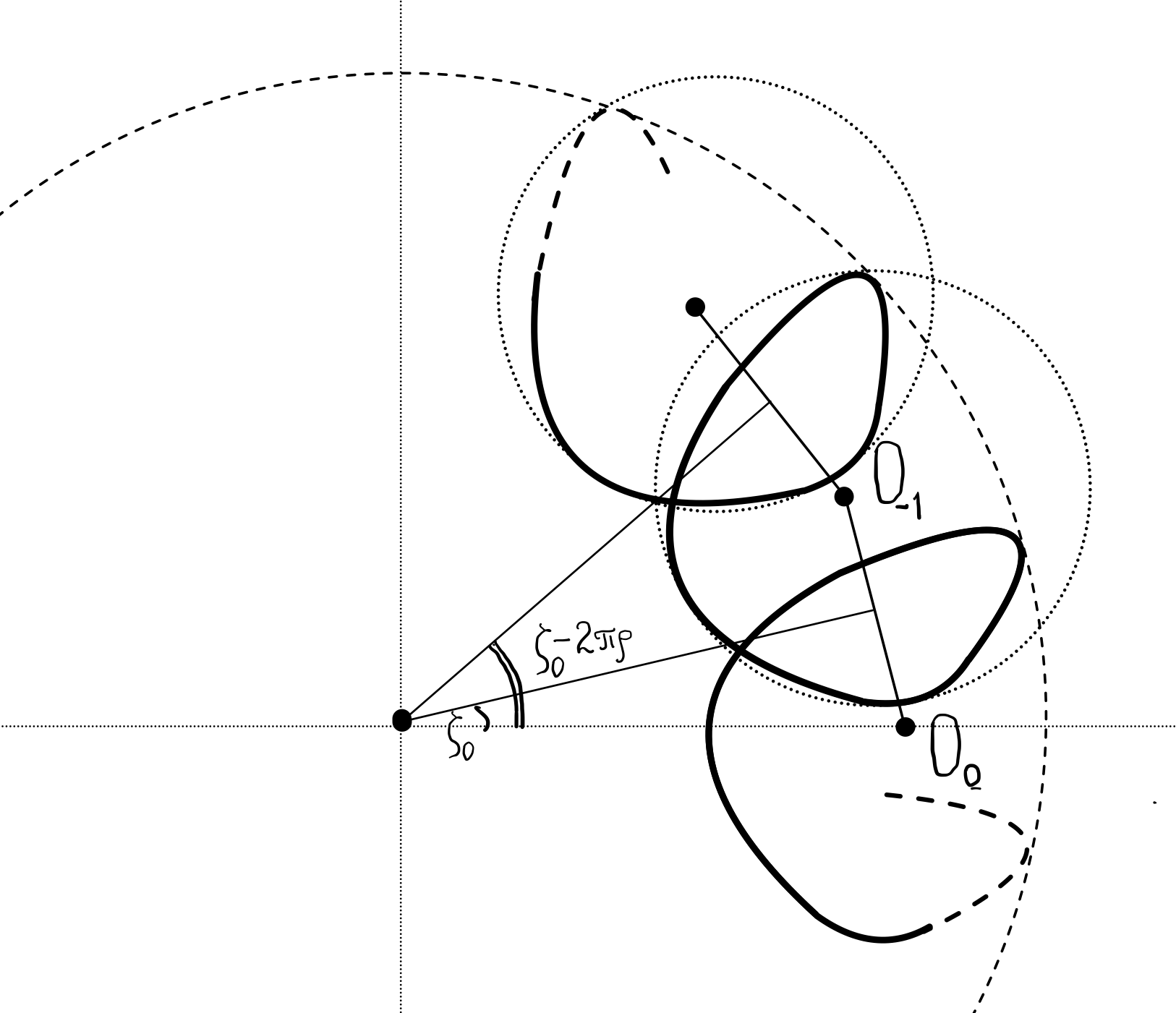}
    \caption{\small Geometrical meaning of the additional mass $\frac{|O_{-1}O_0|}{2\pi}\delta_{\zeta_0/\rho}$ }
    \label{Center}
    \end{figure}
The corresponding measure  is (see Fig.~\ref{Center}):
$$\Delta_0:=\frac{|O_{-1}O_0|}{2\pi}\delta_{\zeta_0/\rho}.$$

The points of $K_h$ where the local surgery should be made are the points $\zeta_j$, where the support line is parallel to the line that connects the two consecutive circumcenters $O_j$ and $O_{j+1}$ (see Fig.~\ref{misparaim}).

\begin{figure}[h]
    \centering
    \includegraphics[width=0.5\linewidth]{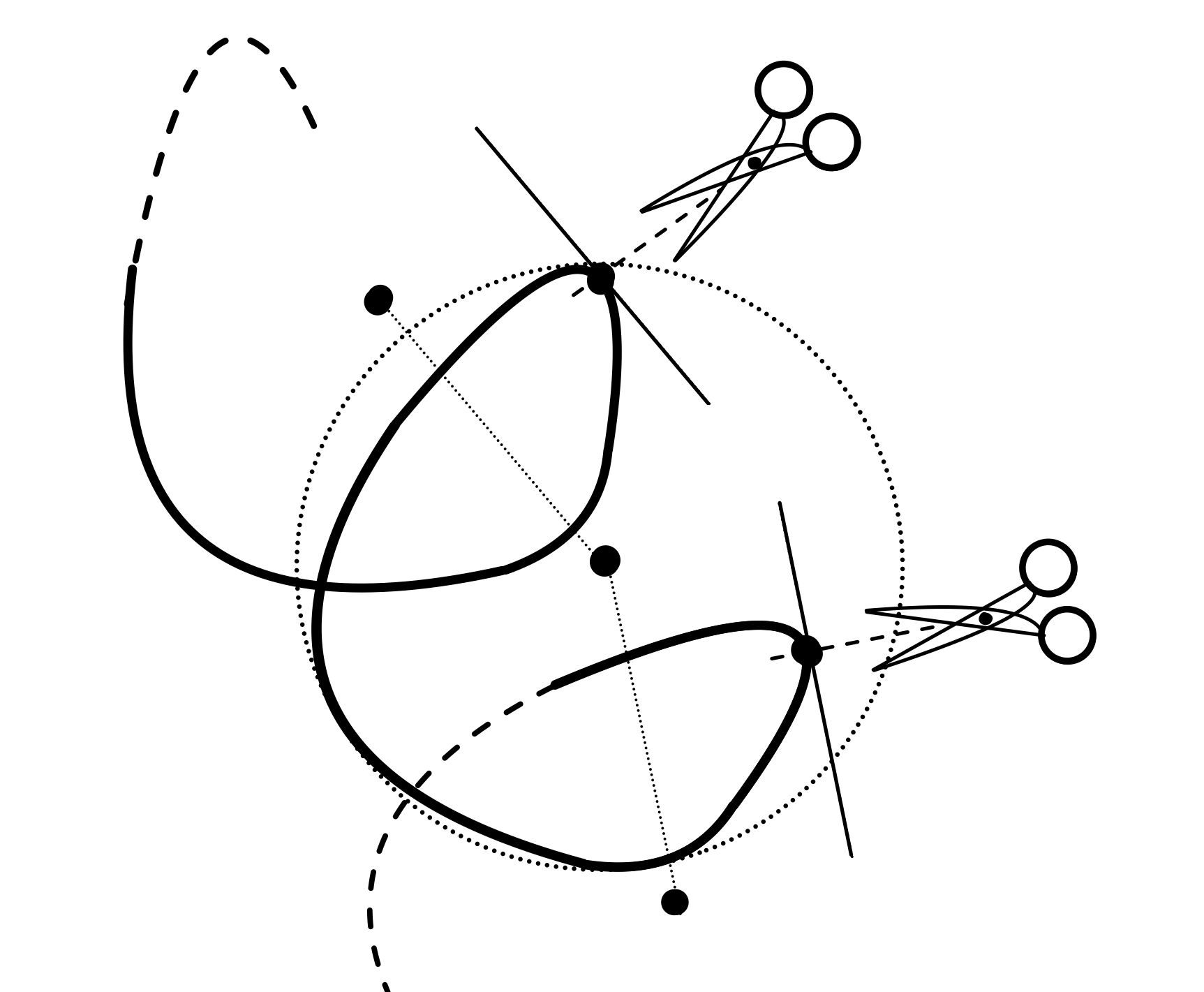}
    \caption{\small Identification of the points where the local surgery should be made.}
    \label{misparaim}
\end{figure}

Note that, 
the surgery points shown in the figure belong to the same orbit under the
rotations involved in the construction. Hence the corresponding local surgeries
represent the insertion of a single additional mass into the measure.

\begin{figure}[H]
    \centering
    \includegraphics[width=0.3\linewidth]{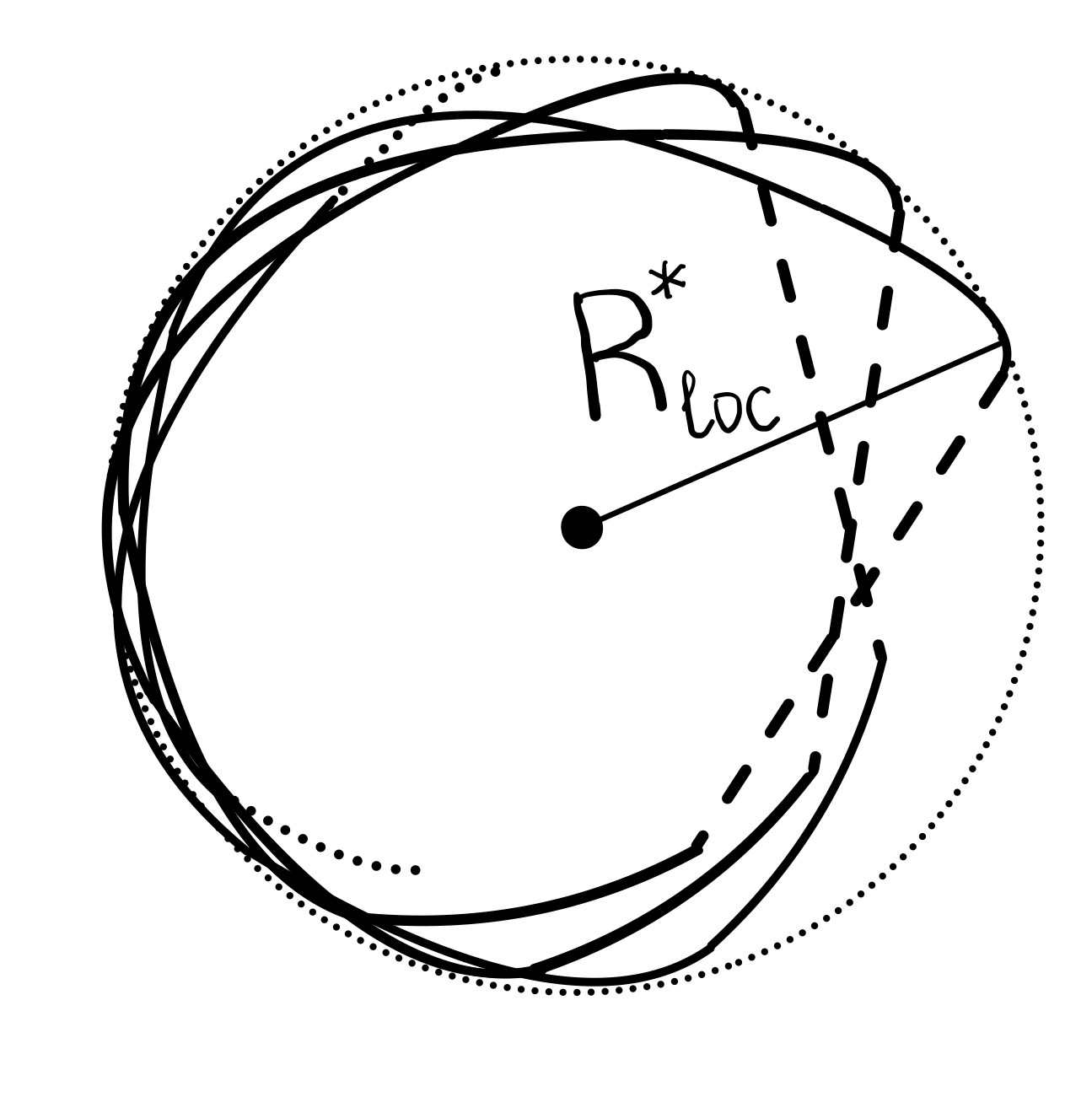}
    \caption{\small Part of the modified locally trigonometrically convex curve $K_h$ after the surgery (the dashed lines correspond to the additional mass).}
    \label{After}
\end{figure}
On Fig.~\ref{After} a part of the modified locally trigonometrically convex curve $K_h$ after the surgery is presented, it fits the circumcircle of the single nest and has circumradius $R_{\rm loc}^*$.

\subsection{ Construction of a type-minimizing measure: the case $\rho=2$}

{
 The case $\rho=2$ occurs to be more technically involved then the case $\rho\in(1/2,1)$ where rotational symmetry plays a crucial simplifying role. The main idea of this section is to show that in the case $\rho=2$ it is also possible to achieve type minimization in geometric way. }

\subsubsection{Covering of $K_h$ by nests}\label{assump_origin}
By Lemma \ref{shift}, we may assume that the circumcenter of $K_h$ is at the origin.
{
Let us recall the definition of a locally $2$-balanced function. Given $h\in TC_2$, put
$$m_h:=\{t\in\mathbb R: h(t/\rho)=\max_{t\in\mathbb R}h(t)=R_h\}.$$
A function $h$ is said to be {locally $2$-balanced}
 if there exist three points $\alpha, \beta,\gamma\in m_h$ such that 
$$0<\beta-\alpha\le \pi,\;\; 0\le \gamma-\beta<\pi, \gamma-\alpha\ge\pi.$$}
The geometric meaning of this notion for the locally trigonometrically convex curve  $K_h$ is  illustrated in Fig.\ref{KI8}.  

Recall, that by item (5) of Theorem \ref{PartI},  $h\in TC_2$ is locally $2$-balanced if and only if $$\sigma_U(\Delta_h)=\sigma_Z(\Delta_h).$$

\begin{figure}[h]
    \centering
    \includegraphics[width=0.9\linewidth]{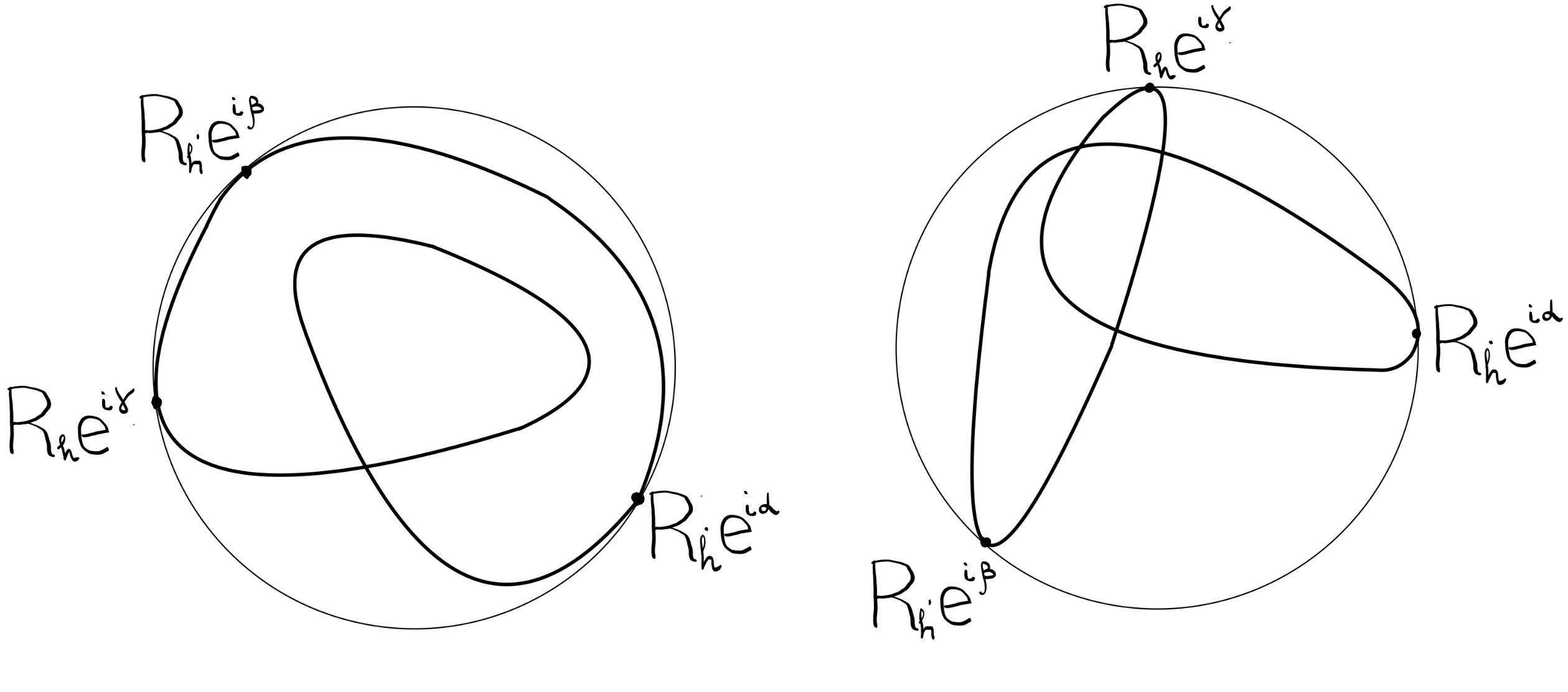}
   \caption{\small Left picture: $K_h$ corresponds to the locally 2-balanced function $h$: { $0<\beta-\alpha\le \pi,\;\; 0\le \gamma-\beta<\pi, \gamma-\alpha\ge\pi.$ }\\
        Right picture: the  corresponding function $h$ is not locally 2-balanced: {  $\beta-\alpha> \pi,\;\;  \gamma-\beta>\pi, (\alpha+4\pi)-\gamma>\pi.$ .}}
      \label{KI8}
\end{figure}

The following lemmas establish key properties of the locally convex curve $K_h$ when $h$ is not locally $2$-balanced.

\begin{lemma}\label{3points}
       If $h$ is not locally $2$-balanced, then 
      for any pair of points $t_1, t_2\in m_h$ we have
      $$(t_1-t_2)\notin\{ \pi(4k+1),\pi(4k+2),\pi(4k+3),\;k\in\mathbb Z\}.$$
       \end{lemma}

\begin{proof}
 Recall that we are considering all the arguments by modulus $4\pi.$ 
 
First, if $t_1-t_2=\pm\pi$, we get a locally $2$-balanced set, that contradicts  to the condition of the lemma.

Now, suppose that $t_1-t_2=2\pi$. Without loss of generality we can assume that $t_1=0, t_2=2\pi$. 
Geometrically the corresponding points of the set $ M_h$ coincide: $R_h e^{0}=R_h e^{2\pi i}=R_h$, so there should be another two points $t_3, t_4\in m_h$ (recall that, by our assumption at the beginning of Subsection~\ref{assump_origin}, the circumcenter of $K_h$ is at the origin).  Using our first conclusion, we see that they do not coincide with the point $R_he^{i\pi}$, hence, their geometric images $R_h e^{it_3},R_h e^{ it_4}$,  together with the point $R_h$, form the vertices of an acute-angled triangle.
Let $0<t_3<t_4< 4\pi.$ Then, using the acuteness of the corresponding triangle, we have
\begin{itemize}
    \item [ ]either $0<t_3<\pi<t_4<2\pi$,
    \item [ ]or $0<t_3<\pi<3\pi<t_4<4\pi$,
    \item [ ]or $\pi<t_3<2\pi<t_4<3\pi$,   
     \item [ ]or $2\pi<t_3<3\pi<t_4<4\pi$.
\end{itemize}

In every case we get a contradiction with the condition that $h$ is not locally 2-balanced.
\end{proof}

\begin{lemma}\label{cover}
If  $  h$ is not locally $2$-balanced, then the corresponding locally convex curve $K_h$ can be covered by three nests  $N_{1},N_{2},N_{3}:$
$$K_h=\bigcup_{j=1}^3 N_{j}.$$

\end{lemma}
\begin{proof}

From the condition  $0\in {\rm conv}(M_h)$,   taking into account the result of the preceding lemma, we get that there exist three  points $A_1,A_2, A_3$ such that 
$A_j=R_he^{i\gamma_j}\in M_h$, where $\gamma_2-\gamma_1\in (\pi,2\pi), $ $\gamma_3-\gamma_2\in (\pi,2\pi), $ and $\gamma_1+4\pi-\gamma_3\in (\pi,2\pi).$

{We shall now apply Lemma \ref{existence of a nest}. 
Recall that the variables $\gamma_j$ are angles in the stretched variable
$t$, where $\widetilde h(t)=h(t/2)$. Thus, in order to apply Lemma
\ref{existence of a nest}, which is stated for the original variable, we use
the points $\gamma_j/2$.

Apply Lemma \ref{existence of a nest} to the following three pairs:
\[
\left(\frac{\gamma_3}{2}-2\pi,\frac{\gamma_1}{2}\right),\qquad
\left(\frac{\gamma_1}{2},\frac{\gamma_2}{2}\right),\qquad
\left(\frac{\gamma_2}{2},\frac{\gamma_3}{2}\right).
\]
Since
\[
\frac{\gamma_2-\gamma_1}{2},\quad
\frac{\gamma_3-\gamma_2}{2},\quad
\frac{\gamma_1+4\pi-\gamma_3}{2}
\in \left(\frac{\pi}{2},\pi\right),
\]
the assumptions of Lemma \ref{existence of a nest} are satisfied for
$\rho=2$. Hence there exist points
\[
t_1^*\in\left[\frac{\gamma_1}{2},\frac{\gamma_3}{2}-\pi\right],
\quad
t_2^*\in\left[\frac{\gamma_2}{2},\frac{\gamma_1}{2}+\pi\right],
\quad
t_3^*\in\left[\frac{\gamma_3}{2},\frac{\gamma_2}{2}+\pi\right],
\]
such that the subarcs
\[
N_{\eta_j}=K_{(\eta_j-2\pi,\eta_j)},\qquad \eta_j:=2t_j^*,
\]
are nests. Equivalently,
\[
\eta_1\in[\gamma_1,\gamma_3-2\pi],\qquad
\eta_2\in[\gamma_2,\gamma_1+2\pi],\qquad
\eta_3\in[\gamma_3,\gamma_2+2\pi].
\]}
Recall that, by the definition of the nest, $K_h$ has the double support line $L_{\eta_j}$, hence, by Lemma \ref{3points} $\eta_j$ cannot coincide with any of the directions $\gamma_j$, $\gamma_j+\pi k.$

We can assume that $\gamma_1=0.$ Then we have
 \begin{multline*}
    \eta_2-2\pi< 0=\gamma_1<\eta_1<\gamma_3-2\pi<\eta_3-2\pi< \gamma_2<\eta_2
    <2\pi= \gamma_1+ 2\pi\\<\eta_1+2\pi<\gamma_3< \eta_3<\gamma_2+ 2\pi< \eta_2
+2\pi<4\pi=\gamma_1+4\pi<\eta_1+4\pi,
\end{multline*}
therefore $ K_h\subset N_{\eta_1}\bigcup N_{\eta_2}\bigcup N_{\eta_3}$  (see Fig.~\ref{KI3}).

\begin{figure}[h]
    \centering
   \includegraphics[width=120mm]{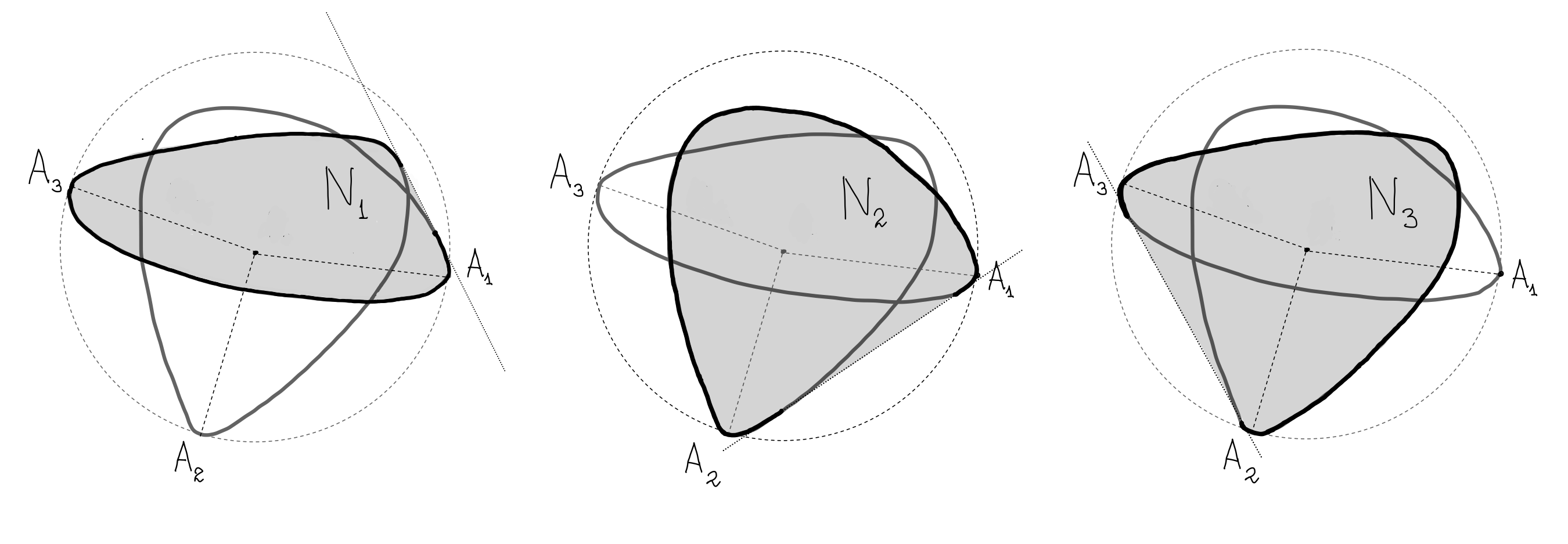}
    \caption{\small Three nests $N_1, N_2, N_3$ that form the covering of the locally convex curve corresponding to $h\in TC_2$ which is  not locally 2-balanced.
}
      \label{KI3}
    
    \end{figure}
    \end{proof}
\begin{rem}
        Note, that the covering of the locally convex curve by nests, given in the proof of  Lemma \ref{cover}, is not necessarily unique.
\end{rem}

\subsubsection{Construction of a type-minimizing measure.}
 Suppose  that the function $h$ is not locally $2$-balanced.  Then item (5) of Theorem  
\ref{PartI} gives
$$\sigma_U(\Delta_h)<\sigma_Z(\Delta_h).$$
Using Theorem \ref{T2} and the identity $\sigma_Z(\Delta_h)=R_h$, we get
 $R_{\rm loc}^*<R_h.$     By Lemma 
\ref{cover}, we have $0\le\gamma_1<\gamma_2<\gamma_3<4\pi,$ where  $\widetilde h(\gamma_j)=\displaystyle \max_{t\in[0,4\pi]} \widetilde h(t)=R_h$, and
$$K_h=\bigcup_{j=1}^3 N_{\eta_j},$$
so that $$K_{[\gamma_{j-1},\gamma_{j}]}\subset N_{\eta_j}\subset D_{j}, \;\;j=1,2,3,$$
with  $\gamma_0:=\gamma_3-4\pi$
and $D_j=D(N_{\eta_j}).$
Then, for every $t\in [\eta_j-2\pi,\eta_{j}]$, we have
$${\widetilde h(t)\le \tau_{j}(t)=R_{j}-r_{j}\cos(t-\theta_{j}),}$$  where $O_j$ is circumcenter of the nest $N_j$, $R_j$ is its circumradius, and the parameters $r_j,\theta_j$ are chosen so that  $$O_{j}=r_je^{i(\theta_{j}+\pi)}=-r_je^{i\theta_{j}}.$$

{In particular, 
we have
\[
\tau_j(\gamma_{j-1})\ge \widetilde h(\gamma_{j-1})=R_h,\qquad \tau_j(\gamma_j)\ge\widetilde h(\gamma_j)= R_h.
\]}

 Let $\widetilde D_j$ be the disk concentric with $D_{j}$  with radius $R^*_{\rm_loc}$.
 Then the support function of the set $\widetilde D_j$ is
 $$\widehat\tau_{j}(t)=R^*_{\rm loc}-r_j\cos(t-\theta_j)\ge \tau_j(t).$$

{
Since $K_{[\gamma_{j-1},\gamma_j]}\subset N_{\eta_j}\subset D_j$ and
\[
\widetilde h(\gamma_{j-1})=\widetilde h(\gamma_j)=R_h,
\]
the disk $D_j$ contains the two corresponding points
$A_{j-1}$ and $A_j$ of the global circumcircle of $K_h$.

We claim that
\[
\tau_j(\gamma_{j-1})>R_h,\qquad \tau_j(\gamma_j)>R_h.
\]
Indeed, suppose that
$\tau_j(\gamma_j)=R_h$, then the support line of $D_j$ in the direction
$\gamma_j$ touches $\partial D_j$ at $A_j$.

Hence the center of $D_j$ is
\[
O_j=(R_h-R_j)e^{i\gamma_j}.
\]
Therefore
\[
|A_{j-1}-O_j|^2-R_j^2
=
2R_h(R_h-R_j)(1-\cos(\gamma_j-\gamma_{j-1}))>0,
\]
because $R_j<R_h$ and $\gamma_j-\gamma_{j-1}\in(\pi,2\pi)$. Thus
$A_{j-1}\notin D_j$, a contradiction. Hence
$\tau_j(\gamma_j)>R_h$. The proof of
$\tau_j(\gamma_{j-1})>R_h$ is identical.

}

Consequently, we have
$$r_j\cos(\gamma_j-\theta_j)< R^*_{\rm loc}-R_h<0,$$
and
$$r_j\cos(\gamma_{j-1}-\theta_j)< R^*_{\rm loc}-R_h<0.$$
Since $\gamma_j-\gamma_{j-1}\in(\pi,2\pi)$, the function $r_j\cos(t-\theta_j)$ attains its maximum on the interval  $(\gamma_{j-1},\gamma_j)$. Without loss of generality we may suppose that $\gamma_{j-1}<\theta_j<\gamma_j$. Thus, there exists $t\in (\theta_j,\gamma_j)$ such that $\widehat \tau_j(t)=R_h.$ 
For $j=1,2,3$ we put 
$$x_j=\max\{t\in(\theta_j,\gamma_j): \widehat\tau_{j}(t)=R_h\}.$$

Analogously, we define 
$$y_j=\min\{t\in(\gamma_j,\theta_{j+1}):\widehat\tau_{j+1}(t)=R_h\}.$$
We use the cyclic notation, putting
\[
x_0:=x_3-4\pi,\qquad y_0:=y_3-4\pi.
\]

Now, since 
 $\widehat\tau_j(t)\le R_{\rm loc}^*<R_h$ for all $t$ such that $|t-\theta_{j}|\le \pi/2,$ while $\widehat\tau_j(\gamma_j)> R_h$
and $\widehat\tau_j(\gamma_{j-1})> R_h$ we conclude that for every  $j=1,2,3$ we have 
  $\pi<x_{j}-y_{j-1}<2\pi.$ It also follows that for every $j=1,2,3$ $y_j-x_j<\pi$.

Let us focus now on the  interval $[x_j,y_j]$. 
At the point $x_j$ the function $\widehat\tau_{j}$ is strictly increasing, and has its maximum at the point $M_{j}:=\theta_{{j}}+\pi>x_j$. Proceeding to the other end of the interval, we see that $\widehat\tau_{j+1} $ decays at the point $y_j$ and  attains maximum at the point $M_{j+1}:=\theta_{j+1}-\pi<y_j$.

Moreover, we have $\widehat\tau_{j}(x_j)<\widehat\tau_{j+1}(x_j)$ and  $\widehat\tau_{j}(y_j)>\widehat\tau_{j+1}(y_j),$ hence the trigonometric function $\widehat\tau_{j}(t)-\widehat\tau_{j+1}(t)\in T_1$ is negative at $x_j$ and positive at $y_j$. Since the interval $[x_j,y_j]$ has length less than $\pi$, this function has a unique zero  $\zeta_j\in(x_j,y_j)$. At this zero it changes sign from negative to positive and therefore  
$$\widehat\tau'_{j}(\zeta_j)-\widehat\tau'_{j+1}(\zeta_j)>0.$$ 

From the geometric point of view, the equality $\widehat\tau_{j}(\zeta_j)=\widehat\tau_{j+1}(\zeta_j)$ means that the line $x\cos\zeta_j+y\sin\zeta_j=\widehat\tau_{j}(\zeta_j)$ is a common tangent line  to $\widetilde D_{j+1}$   and $\widetilde D_j$. By the construction, this tangent line is parallel to the line $(O_{j} O_{j+1})$. 

Set
 $2\pi\mu_j:=\widehat\tau'_{j}(\zeta_j)-\widehat\tau'_{j+1}(\zeta_j)>0.$ 
Using the chosen representation of the centres
$O_{j}=-r_je^{i\theta_{j}}, O_{j+1}=-r_{j+1}e^{i\theta_{j+1}}$, we obtain
$$2\pi\mu_j=r_{{j}}\sin(\zeta_j-\theta_{{j}})-r_{j+1}\sin(\zeta_j-\theta_{{j+1}})={\rm Im}((O_j-O_{j+1})e^{-i\zeta_j}),$$

Since the line $(O_jO_{j+1})$ is parallel to the tangent direction at
angle $\zeta_j$, which is represented by
\(ie^{i\zeta_j}\), the number  $
(O_j-O_{j+1})e^{-i\zeta_j}
$
is purely imaginary. Moreover, its imaginary part is positive by the
inequality above. Therefore
$$
\operatorname{Im}\left((O_j-O_{j+1})e^{-i\zeta_j}\right)
=
|O_jO_{j+1}|.
$$
Consequently,
\[
2\pi\mu_j=|O_jO_{j+1}|.
\]

Put
$$ \widetilde H(t):=
\begin{cases}
 \widehat \tau_1(t), & t\in [\zeta_{3}-4\pi,\zeta_1];\\
 \widehat \tau_2(t), & t\in [\zeta_1,\zeta_2];\\
       \widehat \tau_3(t), & t\in [\zeta_2,\zeta_3];\\
\end{cases} 
$$
our construction yields $\widetilde h(t)\le \widetilde H(t),\;\;\forall t\in[0,4\pi].$
Define
$$ \widetilde k(t):=
R^*_{\rm loc}-\widetilde H(t), \;\;\;k(t):=\widetilde k(2 t),
$$
then $h(t)+k(t)\le R^*_{\rm loc}$.
On the other hand,
$$\widetilde k(t):=
\begin{cases}
r_1\cos(t-\theta_{1}) & t\in [\zeta_3-4\pi,\zeta_1];\\
  r_2\cos(t-\theta_{2}) & t\in [\zeta_1,\zeta_2];\\
   r_3\cos(t-\theta_{3}), & t\in [\zeta_{2},\zeta_3],
    \end{cases}, 
$$
The corresponding  $2$-trigonometrically convex function  $k(t)=\widetilde k(2 t)$ is associated with the measure 
$$\Delta_0:=\sum_{j=1}^3\mu_j\delta_{\zeta_{j}/2}.$$

Looking from the geometric point of view, the operation of adding $k(t)$ to $h(t)$ corresponds to the following local surgery of $K_h$: we cut $K_h$ at the points  $\zeta_{j}$, where we have $L_{\zeta_{j}}\| (O_{j}O_{j+1})$, and insert  segments of length $2\pi\mu_j=|O_jO_{j+1}|$.

\subsection*{Acknowledgements}

The author is deeply grateful to A. Borichev and M. Sodin for their continuous support and detailed feedback throughout the preparation of this work.

The author also gratefully acknowledges Tel Aviv University, where a substantial part of this work was carried out.

This work was supported by Israel Science Foundation grant No. 1288/21 and by the Center for Integration in Science of the Ministry of Aliyah and Integration of Israel. The author gratefully acknowledges this support.

\bigskip
\end{document}